\documentclass[11pt,reqno]{amsart}
\usepackage[text={400pt,660pt},centering]{geometry}
\usepackage{color,xcolor}
\usepackage{esint,amssymb}
\usepackage{graphicx}
\usepackage{MnSymbol}
\usepackage{mathtools}
\usepackage[colorlinks=true, pdfstartview=FitV, linkcolor=blue, citecolor=blue, urlcolor=blue,pagebackref=false]{hyperref}

\usepackage{bm}
\usepackage{scalerel} %for scaling the boxes 
\usepackage{dsfont}
\usepackage{mathrsfs}
\usepackage{eufrak}
\usepackage[font={footnotesize}]{caption}
\usepackage[shortlabels]{enumitem}
\usepackage[normalem]{ulem}

\usepackage{tikz-qtree}
\usepackage{bbm}
\usepackage{pgfplots}
\usetikzlibrary{decorations.pathmorphing}
\usepgfplotslibrary{fillbetween}
\usetikzlibrary{intersections,patterns}

\usepackage{soul}
\definecolor{clou}{rgb}{0.8,0.25,0.5125}
\reversemarginpar
\renewcommand{\emptyset}{\varnothing}

\renewcommand{\P}{\mathbf{P}}

\newcommand{\R}{\mathbb{R}}
\newcommand{\RR}{\mathbb{R}}

\newcommand{\indc}{\mathds{1}}
\newcommand{\Z}{\mathbb{Z}}
\newcommand{\norm}[1]{\left|\left|#1\right|\right|}
\newcommand{\N}{\mathbb{N}}

\DeclareMathOperator{\Bet}{Beta}

\numberwithin{equation}{section}

\newtheorem{thm}{Theorem}[section]
\newtheorem{lemma}[thm]{Lemma}
\newtheorem{prop}[thm]{Proposition}
\newtheorem{cor}[thm]{Corollary}

\theoremstyle{remark}
\newtheorem{remark}[thm]{Remark}
\theoremstyle{definition}
\newtheorem{define}[thm]{Definition}

\newcommand{\ve}{\varepsilon}

\newcommand{\vp}{\varphi}

\newcommand{\La}{\Lambda}

\newcommand{\Ga}{\Gamma}

\newcommand{\B}{{\mathbf B}}

\newcommand{\p}[1]{{\mathbf P}\left(#1\right)}

\newcommand{\pran}[1]{\left(#1\right)}
\newcommand\restr[2]{{% we make the whole thing an ordinary symbol
  \left.\kern-\nulldelimiterspace % automatically resize the bar with \right
  #1 % the function
  \vphantom{\big|} % pretend it's a little taller at normal size
  \right|_{#2} % this is the delimiter
  }}

\newcommand{\scmmu}{\ensuremath{\mathrm{SCM}(m,\mu)}}

\newcommand{\tmesh}{\Delta_t}
\newcommand{\xmesh}{\Delta_x}
\newcommand{\eps}{\varepsilon}

\title[]{Symmetric Cooperative motion in one dimension}
\author[]{Louigi Addario-Berry}

\address{Department of Mathematics and Statistics, McGill University}
\email{louigi.addario@mcgill.ca}

\author[]{Erin Beckman}
\address{Department of Mathematics and Statistics, McGill University}
\email{erin.beckman@mcgill.ca}

\author[]{Jessica Lin}
\address{Department of Mathematics and Statistics, McGill University}
\email{jessica.lin@mcgill.ca}

\pgfplotsset{compat=1.14}
\begin{document}


\begin{abstract}
We explore the relationship between recursive distributional equations and convergence results for finite difference schemes of parabolic partial differential equations (PDEs). We focus on a family of random processes called symmetric cooperative motions, which generalize the symmetric simple random walk and the symmetric hipster random walk introduced in \cite{Addario-Berry2019}. We obtain a distributional convergence result for symmetric cooperative motions and, along the way, obtain a novel proof of the Bernoulli central limit theorem. In addition, we prove a PDE result relating distributional solutions and viscosity solutions of the porous medium equation and the parabolic $p$-Laplace equation, respectively, in one dimension.
\end{abstract}

\maketitle
\section{Introduction}\label{s.intro}
\subsection{Description of the model and the main result.} We consider a family of random processes called symmetric cooperative motions, which generalize the symmetric simple random walk on $\Z$. 
Informally, symmetric cooperative motion is a random walk where at each step, the walker requires the assistance of $m$ other walkers (independent copies of the process) in order to move. If all $m$ copies are at the same location as the walker, the walker can take a step. Otherwise, it must stay put.

The model can be constructed in two ways. The first, which is the primary construction we will work with, is as follows. Given an initial distribution $\mu$ on $\Z\cup\{-\infty,\infty\}$, we define $\left(X_{n}, n\geq 0\right)$ via the following recursive distributional equation (RDE): 
\begin{equation}\label{eq:main_recurrence}
X_{n+1} := \begin{cases} 
			X_n + E_n	& \mbox{ if } X_n=X_{n,1}=\ldots=X_{n,m} \,  , \\
			X_n	&\mbox{ otherwise.} 
			\end{cases}
\end{equation}
Here $(E_n,n \ge 0)$ are IID with $\p{E_0=1}=1/2=\p{E_0=-1}$, and $X_{n,1},\ldots,X_{n,m}$ are IID copies of $X_n$.  
We set $-\infty+x=-\infty$ and $\infty+x=\infty$ for $x \in \R$; allowing the initial condition to take values $\pm \infty$ will be useful in upcoming sections. 
We write $\scmmu$ for the law of the process $(X_n,n \ge 0)$ defined by \eqref{eq:main_recurrence} when $X_0$ has distribution $\mu$. 

The second construction interprets symmetric cooperative motion as a tree-indexed random process. Let $\mathcal{T}$ be the complete rooted $(m+1)$-ary tree, with root labeled by $\emptyset$ and node $v$ having children $(vi, 1\leq i \leq m+1)$. Let $\mathcal{T}_{n}$ denote the subtree of $\mathcal{T}$ containing nodes at distance at most $n$ from the root, and let $\mathcal{L}_{n}$ denote the leaves of $\mathcal{T}_{n}$. 

In order not to cause confusion, we will use a different set of variables (rather than $X_{n}$) to define the cooperative motion in this setting. Fix a probability distribution $\mu$ on $\Z\cup\{-\infty,\infty\}$ and for each $n\in \N$, let $\Sigma^{n}=(\Sigma^{n}_{v}, v\in \mathcal{L}_{n})$ be a collection of IID, $\mu$-distributed random variables. Let $E^n=(E_{v}, v\in \mathcal{T}_{n}\setminus \mathcal{L}_{n})$ denote a separate collection of IID Bernoulli($\tfrac{1}{2}$) random variables, independent of $\Sigma^n$, and for 
$v\in \mathcal{T}_{n}\setminus \mathcal{L}_{n}$, recursively define 
\begin{equation*}
\Sigma^{n}_{v}:=\begin{cases} \Sigma^{n}_{v1}+E_{v}&\text{if $\Sigma_{v1}=\Sigma_{v2}=\ldots=\Sigma_{v(m+1)}$},\\
\Sigma^{n}_{v1}&\text{otherwise.}
\end{cases}
\end{equation*}
This recursion yields an ``output value'' $\Sigma^n_{\emptyset}$ at the root, which has the same distribution as $X_n$ from the first construction. (Note: the {\em processes} $(\Sigma_\emptyset^n,n \ge 0)$ and $(X_n,n \ge 0)$ have different laws, but their one-dimensional distributions are the same.)

Returning to our notation $(X_{n}, n\geq 0)$ as defined by \eqref{eq:main_recurrence}, for $p^n_k := \p{X_n=k}$, we have 
\begin{equation}\label{e.pgen}
p^{n+1}_k = p^n_k(1-(p^n_k)^m) + \frac{1}{2} (p^n_{k+1})^{m+1} + \frac{1}{2} (p^n_{k-1})^{m+1}\, .
\end{equation}
The first term, $p^n_k(1-(p^n_k)^m)$, corresponds to the event that $X_n=k$ and at least one of $X_{n,1},\ldots,X_{n,m}$ is different from $k$. The second term corresponds to the event that $X_n,X_{n,1},\ldots,X_{n,m}$ all equal $k+1$ and $E_n=-1$, and the third term corresponds to the event that $X_n,X_{n,1},\ldots,X_{n,m}$ all equal $k-1$ and $E_n=+1$. Rearranging gives 
\[
p^{n+1}_k - p^n_k-
\frac{1}{2}\pran{
(p^n_{k+1})^{m+1}-2(p^n_{k})^{m+1}+(p^n_{k-1})^{m+1}}=0. 
\]

In the case when $m>0$ is non-integer, setting $p^0_k = \mu(\{k\})$ and defining $p^n_k$ inductively using \eqref{e.pgen}, it still holds that $\sum_{k \in \Z} p^n_k = 1$ for all $n$; we may therefore define a sequence of random variables $(X_n,n \ge 0)$ with $\p{X_n=k}=p^n_k$, and study the asymptotic behaviour of $X_n$ even when $m$ is non-integer. However, in this case the interpretations in terms of the recurrence (\ref{eq:main_recurrence}) and in terms of the tree-indexed process are unavailable.

Our main result describes the asymptotic distributional behaviour of $X_n$.
\begin{thm}\label{t.main} Fix $m>0$, and any probability distribution $\mu$ on $\Z$, and let $(X_n,n \ge 0)$ be \scmmu-distributed. Let 
\[
\rho:=\frac{m}{2(m+2)}\quad\mbox{ and }\quad 
D:=\rho^{-1/2}\B\left(\frac{1}{2}, \frac{m+1}{m}\right)\, ,
\]
where $\B(\cdot \,, \cdot)$ denotes the Beta Function. 
Then 
\begin{equation}\label{e.distconv}
\frac{X_n}{n^{1/(m+2)}}\xrightarrow{d} \frac{2^{\frac{m+1}{m+2}}(m+1)^{1/(m+2)}}{D^{\frac{m}{m+2}}\rho^{1/2}}\left(B-\frac{1}{2}\right),
\end{equation}
where $B$ is $\Bet\left(\frac{m+1}{m},\frac{m+1}{m}\right)$-distributed. 
\end{thm}
As explained in \cite{Addario-Berry2019}, in the case $m=1$ there is a compelling heuristic connection between the asymptotic behaviour of symmetric cooperative motion and that of the {\em critical random hierarchical lattice}, a model of random electrical network studied by Hambly and Jordan \cite{MR2079916}. This connection yields new predictions (but thus far, no proofs) for the scaling behaviour of the effective resistance in such random networks.

\subsection*{An aside: the $m \to 0$ limit of symmetric cooperative motion} 
This short section relates the asymptotic distribution of $X_n$ to that seen for a symmetric simple random walk (SSRW); it is not necessary for the rest of the paper. 

Taking $m=0$ in the cooperative motion dynamics, the condition that $X_n=X_{n,1}=\ldots=X_{n,m}$ becomes vacuous. Thus, the process makes a random move at every step, and we recover the SSRW dynamics. In this sense, one may say that at the level of processes, the $m\to 0$ limit of the cooperative motion is just the symmetric simple random walk. 

This $m \to 0$ limit is also reflected in the behaviour of the limiting distributions. 
To see this, rewrite 
\begin{align*}{}{}
\frac{1}{D^{\frac{m}{m+2}}\rho^{1/2}}
&\left[\Bet\left(\frac{m+1}{m}, \frac{m+1}{m}\right)-\frac{1}{2}\right] \\
&=\frac{\rho^{\frac{m}{2(m+2)}}}{\textbf{B}\left(\frac{1}{2}, \frac{m+1}{m}\right)^{\frac{m}{m+2}}}\frac{1}{\rho^{1/2}}\left[\Bet\left(\frac{m+1}{m}, \frac{m+1}{m}\right)-\frac{1}{2}\right] \\
&=\frac{1}{\rho^{\frac{1}{m+2}}\textbf{B}\left(\frac{1}{2}, \frac{m+1}{m}\right)^{\frac{m}{m+2}}}\left[\Bet\left(\frac{m+1}{m}, \frac{m+1}{m}\right)-\frac{1}{2}\right]
\end{align*}

By Stirling's formula, since $\frac{m+1}{m}\xrightarrow[m\to 0]{}\infty$, we have 
\begin{equation*}
\textbf{B}\left(\frac{1}{2}, \frac{m+1}{m}\right)=\left[1+o(1)\right]\Ga\left(\frac{1}{2}\right)\sqrt{\frac{m}{m+1}}.
\end{equation*}
Continuing from the prior computation, and using the definition of $\rho$, we then have 
\begin{align*}
\frac{1}{D^{\frac{m}{m+2}}\rho^{1/2}}&\left[\Bet\left(\frac{m+1}{m}, \frac{m+1}{m}\right)-\frac{1}{2}\right]\\
&=\frac{1+o(1)}{\left[\frac{m}{2(m+2)}\right]^{\frac{1}{m+2}}\left[\Ga\left(\frac{1}{2}\right)\sqrt{\frac{m}{m+1}}\right]^{\frac{m}{m+2}}}\left[\Bet\left(\frac{m+1}{m}, \frac{m+1}{m}\right)-\frac{1}{2}\right].
\end{align*}
We thus obtain the distributional limit
\begin{align*}
&\frac{2^{\frac{m+1}{m+2}}(m+1)^{1/(m+2)}}{D^{\frac{m}{m+2}}\rho^{1/2}}\left[\Bet\left(\frac{m+1}{m}, \frac{m+1}{m}\right)-\frac{1}{2}\right]\\
&=\left[1+o(1)\right]\frac{2^{\frac{m+1}{m+2}}(m+1)^{1/(m+2)}}{\left[\frac{m}{2(m+2)}\right]^{\frac{1}{m+2}}\left[\Ga\left(\frac{1}{2}\right)\sqrt{\frac{m}{m+1}}\right]^{\frac{m}{m+2}}}\left[\Bet\left(\frac{m+1}{m}, \frac{m+1}{m}\right)-\frac{1}{2}\right]\\
&=\left[1+o(1)\right]2^{\frac{m+1}{m+2}}\left[2(m+2)\right]^{\frac{1}{m+2}}\sqrt{\frac{m+1}{m}}\left[\Bet\left(\frac{m+1}{m}, \frac{m+1}{m}\right)-\frac{1}{2}\right]\\
&
\stackrel{d}{\longrightarrow}
\mathcal{W}_{1}
\end{align*}
as $m \to 0$, where $\mathcal{W}_{1}$ is a standard Gaussian; here we have used the fact that 
\begin{equation*}
\sqrt{8n}\left[\Bet(n,n)-\frac{1}{2}\right]\stackrel{d}{\longrightarrow}\mathcal{W}_{1}
\end{equation*}
as $n \to \infty$. (We use the notation $\mathcal{W}_1$ as later $\mathcal{W}_t$ will denote a centered Gaussian with variance $t$.)

\subsection{The main ideas, related work, and an overview of the paper.}
We begin with a brief introduction of the main ideas of the proof of Theorem \ref{t.main}. Our proof technique is based on the observation that the recurrence \eqref{e.pgen} looks like a discrete approximation of a PDE: the {\em porous medium equation} 
\begin{equation}\label{pme_entropy}
u_t -\frac{1}{2} (u^{m+1})_{xx}=0\, .
\end{equation}
Letting $F^n_{k} := \p{X_n \le k}=\sum_{j=-\infty}^{k}p^n_{j}$, \eqref{e.pgen} then also yields a recurrence relation for $F^{n}_{k}$. Indeed, we have
\begin{align*}
F^{n+1}_{k} &=\sum_{j=-\infty}^{k}p^{n+1}_{j}\\
&=\sum_{j=-\infty}^{k} p^n_j(1-(p^n_j)^m) + \frac{1}{2} (p^n_{j+1})^{m+1} + \frac{1}{2} (p^n_{j-1})^{m+1}\\
&=F^{n}_{k}+\sum_{j=-\infty}^{k} -(p^n_j)^{m+1}+ \frac{1}{2} (p^n_{j+1})^{m+1} + \frac{1}{2} (p^n_{j-1})^{m+1}\\
&=F^{n}_{k}+\sum_{j=-\infty}^{k} \left(\frac{1}{2}\left[(p^{n}_{j+1})^{m+1}-(p^{n}_{j})^{m+1}\right]-\frac{1}{2}\left[(p^{n}_{j})^{m+1}-(p^{n}_{j-1})^{m+1}\right]\right)\\
&=F^{n}_{k}+\frac{1}{2}(p^{n}_{k+1})^{m+1}-\frac{1}{2}(p^{n}_{k})^{m+1}.
\end{align*}
Rearranging and using that $p^{n}_{k}=F^{n}_{k}-F^{n}_{k-1}$, we obtain 
\begin{equation}\label{pme_viscosity}
F^{n+1}_k - F^n_k -\frac{1}{2}\left[(F^n_{k+1} - F^n_{k})^{m+1} - (F^n_k - F^n_{k-1})^{m+1}\right]=0.
\end{equation}
We can think of this as a discretization of the PDE 
\begin{equation}\label{eq:integrated}
v_t -\frac{1}{2}\left((v_x)^{m+1}\right)_x=0\, ,
\end{equation}
or equivalently $v_t - \tfrac{m+1}{2} v_x^mv_{xx}=0$.
When $v_{x}\geq 0$, this PDE can be rewritten in the form \begin{equation}\label{e.normplaplace}
v_{t}-\frac{1}{2}\left(|v_{x}|^{m}v_{x}\right)_{x}=0 \, , 
\end{equation}
or $v_{t}-\tfrac{m+1}{2}|v_{x}|^{m}v_{xx}=0$,
which is called the {\em parabolic $p$-Laplace equation}.
We can think of \eqref{eq:integrated} as an ``integrated version'' of the porous medium equation, since if $u$ solves \eqref{pme_entropy} in the classical sense, then the antiderivative 
\begin{equation}\label{e.anti}
v(x,t):=\int_{-\infty}^{x} u(y,t)\, dy
\end{equation}
formally solves \eqref{eq:integrated}. We will focus our analysis primarily on \eqref{pme_viscosity} since the integration in \eqref{e.anti} leads us to expect that $v$ should have an additional (higher) order of regularity in space compared to $u$.

As in our previous paper \cite{MR4391736}, our main approach to establish distributional convergence will be to use convergence results for finite difference approximation schemes of nonlinear PDEs in order to demonstrate convergence of the rescaled CDFs to a continuous function. In the setting of this paper, the CDFs of the rescaled process
\begin{equation}\label{eq:rescaled_process}
\left(\frac{X_t}{n^{1/(m+2)}}, 0 \le t \le n\right)\, ,
\end{equation}
converge to the solution of an initial value problem corresponding to the PDEs \eqref{eq:integrated} and/or \eqref{e.normplaplace} in the viscosity sense.

At this point, a word about initial conditions is in order. If the random variable $X_0$ is $\mu$-distributed then the ``initial condition'' of the sequence of CDFs is given by $F^0_k=\P(X_0 \le k)=\mu(-\infty,k]$. However, if we rescale as in \eqref{eq:rescaled_process}, then the initial condition of the rescaled process is $X_0 / n^{1/(m+2)}$, and the CDF of this random variable  approaches a Heaviside function as $n \to \infty$. For this reason, to connect the probabilistic evolution with that of PDE \eqref{e.normplaplace}, a Heaviside initial condition for the PDE is required.

In \cite{MR4391736}, we considered a variant of the above dynamics, where the random variables $(E_n,n \ge 0)$ are not symmetric, but instead have a bias to the right or left. In the non-symmetric case, the recurrence relation for $F^n_k$ looks like a discretization of a different PDE known as a {\em Hamilton-Jacobi equation}, $v_t+\sigma|v_x|^{m+1}=0$, for an appropriate value $\sigma$ depending on the parameters in the distribution of the random walk steps. 
In the symmetric case which we address here, several challenges arise compared to the setting of \cite{MR4391736}, of which the following two are notable: 
\begin{itemize}
\item The theory of lower semicontinuous viscosity solutions, which was required throughout \cite{MR4391736} to handle the Heaviside initial condition appearing in the limit, does not exist for degenerate parabolic equations. While there have been some attempts to develop a theory of discontinuous viscosity solutions for parabolic equations such as \eqref{e.normplaplace} (see for example \cite{ishii}), these theories generally do not produce unique solutions. 
\item In \cite{MR4391736}, we frequently used the Hopf-Lax formula to understand properties of solutions to the relevant equations. However, in this setting, there is no natural, physically motivated, explicit representation formula for a ``candidate viscosity solution'' of \eqref{e.normplaplace} with Heaviside initial conditions. \end{itemize} 

To overcome these challenges, we use a combination of techniques from nonlinear PDEs to approximate, via continuous viscosity solutions, what we believe to be the ``physically correct'' viscosity solution of the parabolic $p$-Laplace equation with Heaviside initial conditions. Our approach is essentially a classical ``vanishing viscosity'' argument, which was the original motivation for the definitions appearing in the modern theory of viscosity solutions. It turns out that this is enough to identify the limiting behaviour of the rescaled symmetric cooperative motions, as we are able to obtain an explicit formula for the continuous viscosity solutions which approximate our desired solution. A key ingredient of our argument is to rigorously relate solutions of the porous medium equation \eqref{pme_entropy} and those of the parabolic $p$-Laplace equation \eqref{e.normplaplace}. In essence, we prove that $u$ is a distributional solution of \eqref{pme_entropy} if and only if $v$ is a viscosity solution of \eqref{e.normplaplace} (see Theorem \ref{t.id} for a more precise statement). To our knowledge, this is a novel result for  degenerate parabolic equations in one spatial dimension. 

This paper is part of a broader attempt to understand the relationship between RDEs and convergence results for finite difference schemes of PDEs. We focus on illustrating our method in the proof of Theorem \ref{t.main}, and as an expository example,we use our techniques to provide a novel proof of the Bernoulli central limit theorem in Section \ref{sec:casestudy}. Our approach builds upon that used in our prior work \cite{MR4391736}, but we believe we now have a clearer perspective on the method and its robustness. In Section \ref{s.robust} we present a framework and summary of how to use this method to obtain distributional convergence results for other discrete-time, $\R$-valued random processes.

The structure of the remainder of the paper is as follows. In Section \ref{s.bs}, we present a convergence result of Barles and Souganidis \cite{BS} for finite difference schemes of degenerate parabolic PDEs; this will be the main technical input from the theory of numerical approximation of PDEs used throughout our paper. In order to illustrate the main ideas we explain how our approach applies in the setting of simple random walk in Section \ref{sec:casestudy}; this yields a novel proof of the Bernoulli central limit theorem. Section \ref{sec:bscm} then contains a discussion of the challenges in adapting this method to symmetric cooperative motion. In Section \ref{s:ZKB}, we establish the ingredients from PDEs which will be needed for our analysis. These include a discussion of the Zakharov-Kuznetsov-Burgers (ZKB) solution of the porous medium equation, as well as the statement and proof of Theorem \ref{t.id}, which relates distributional solutions and viscosity solutions of certain degenerate parabolic PDEs. This section relies on Appendix~\ref{s.app}, in which we review the notions of PDE solutions we work with throughout the paper. Sections~\ref{sect:LipIC},~\ref{sec:conv_pbounded} and~\ref{s.removep*}  then contain the key technical ingredients, culminating in the proof of Theorem \ref{t.main} in Section \ref{ss.proofofthm}. We conclude our paper in Section \ref{s.conclusion} with a ``recipe'' for how to apply our method in other settings. This section also 
presents several additional results whose proofs are straightforward adaptations of those of Theorem \ref{t.main} and of analogous results from \cite{MR4391736}. 

\subsection{Definitions and Notation}
Given a Euclidean space $\mathcal{X}$, we use $L^{p}(\mathcal{X})$ to denote the classical $L^{p}$ space equipped with the norm $\norm{f}_{L^{p}(\mathcal{X})}=\left(\int_{\mathcal{X}}|f|^{p}\right)^{1/p}.$ We say that $f\in C^{k}(\mathcal{X})$ for $k\in \mathbb{N}$ if $f$ is $k$-th order continuously differentiable, and $f\in C^{k}(\mathcal{X})$ for $k\in (0, \infty)$ if $f\in C^{\lfloor k\rfloor}(\mathcal{X})$ and the $\lfloor k\rfloor$-th order derivative belongs to the H\"older space $C^{k-\lfloor k\rfloor}(\mathcal{X})$. We reserve the notation $f\in C^{k,j}(\mathcal{X})$, for $k,j\in (0, \infty)$, whenever $f$ depends on space and time (so, e.g., $\mathcal{X}=\R\times[0,T]$), with $f\in C^{k}$ in the spatial variables and $f\in C^{j}$ in the time variable. We denote $f\in C^{\infty}_{c}(\mathcal{X})$ if $f$ is an infinitely differentiable function with compact support. We say that $u: \R\times [0, \infty)\rightarrow \R$ belongs to $C([0,T]; L^{p}(\R))$ if $u(\cdot,t)\in L^{p}(\R)$ for every $t\in [0,T]$ and $u(\cdot, t)$ is continuous in $t$ with respect to the $L^{p}(\R)$-norm.

\section{The Barles-Souganidis approximation framework }\label{s.bs}

\subsection{Introducing the framework}\label{sec:introducingframework}
The main external input to our proof is a result of Barles and Souganidis \cite{BS} on the convergence of approximation schemes of viscosity solutions for second-order PDEs, and the goal of this section is to present that result (Theorem~\ref{t.bs2}, below) in the restricted setting we use here. The definition of viscosity solutions, and facts about viscosity solutions relevant to this paper and in particular to Theorem~\ref{t.bs2}, appear in Appendix~\ref{s.app}. 

We consider initial value problems (IVPs) of the following form
\begin{equation}\label{e.ivp2}
\begin{cases}
v_{t}+G(v_x,v_{xx})=0&\text{in $\RR\times (0, \infty)$,}\\
v(x,0)=v_{0}(x)&\text{in $\RR$}\, ,
\end{cases}
\end{equation}
where $G: \R \times \R \to \R$ and $v_0: \R \to \R$.We refer to $v_0$ as the initial condition of \eqref{e.ivp2}.
We say that the initial value problem \eqref{e.ivp2} is {\em good} if the following conditions hold:
\begin{enumerate}
\item $G$ is continuous.
\item $G$ is degenerate elliptic, i.e., $G(p,a) \le G(p,b)$ if $a \ge b$, for all $p \in \R$.
\item $v_0$ is bounded and uniformly continuous.
\end{enumerate}
Under these conditions, it is known (see \cite{users}, Section 5 and the discussion on page 50) that the initial value problem satisfies {\em strong uniqueness}: there exists a unique continuous viscosity solution $v$ of \eqref{e.ivp2}, and moreover, if $v^-$ is any upper semicontinuous subsolution of \eqref{e.ivp2} and $v^+$ is any lower semicontinuous supersolution of \eqref{e.ivp2}, then $v^- \le v^+$.

We next consider {\em approximation schemes} for \eqref{e.ivp2}. For each $N\in (0,\infty)$, fix time and space mesh sizes $\tmesh^N,\xmesh^N > 0$ which, in our applications, will always go to $0$ as $N \to \infty$, and fix a function $\mathcal{G}^N:\R^3 \to \R$. 
Introducing the abbreviation 
$\langle f(x,t)\rangle_N := (f(x-\xmesh^N,t),f(x,t),f(x+\xmesh^N,t))$, we define a finite difference scheme for 
\eqref{e.ivp2} by 
\begin{equation}\label{e.approxscheme}
\begin{cases}
\frac{v^{N}(x,t+\tmesh^N)-v^{N}(x,t)}{\tmesh^N} +\mathcal{G}^{N}\langle v^{N}(x,t)\rangle_N=0, & \text{in $\R\times [\tmesh^{N}, \infty)$}\, , \\
v^{N}(x,0)=v_{0}\big(\big\lfloor \frac{x}{\xmesh^N}\big\rfloor \xmesh^N\big), 
& \text{in $\R\times [0, \tmesh^N)$}\, .
\end{cases}
\end{equation}
We also refer to $v_0$ as the initial condition of \eqref{e.approxscheme}. 
For each $N>0$, once we choose the initial condition $v_0$, the approximation scheme \eqref{e.approxscheme} uniquely determines the function $v^N: \R\times[0,\infty) \to \R$. However, in what follows we will often need to compare the functions defined by \eqref{e.approxscheme} for various initial conditions. As such, we write $u^N,v^N$ and $w^N$ to indicate the functions defined by \eqref{e.approxscheme} when started from initial conditions $u_0,v_0,w_0:\R \to \R$, respectively.

Now fix a family $\mathcal{C}\subset L^\infty(\R)$ with $v_0 \in \mathcal{C}$. 
We say that \eqref{e.approxscheme} is a {\em good approximation scheme} for \eqref{e.ivp2} on $\mathcal{C}$ if the following conditions hold. 
\begin{enumerate}
\item {\bf Monotonicity:} 
For all $N$ sufficiently large and all $u_0,w_0 \in \mathcal{C}$, 
\begin{equation}\label{e.gmon}
u_0 \le w_0 \mbox{ on }\R \Rightarrow u^N \le w^N \mbox{ on } \R\times[0,\infty)\, .
\end{equation}
\item {\bf Stability:} $\limsup_{N \to \infty} \sup_{(x,t) \in \R\times(0,\infty)} |v^N(x,t)| < \infty$.
\item {\bf Consistency:} If $\varphi:\R \times(0,\infty) \to \R$ is any bounded smooth function, then for all $(x,t) \in \R \times (0,\infty)$, 
\[
\lim_{(N,y,s,\varepsilon) \to (\infty,x,t,0)} \mathcal{G}^{N}\langle \varphi(y,s)+\eps\rangle_N = G(\varphi_x(x,t),\varphi_{xx}(x,t))\, .
\]
\end{enumerate}
In the last point, we use the notation $\lim_{(N,y,s,\varepsilon) \to (\infty,x,t,0)}$ to indicate that the limit may be taken in any order.

The first point, the monotonicity condition, plays a crucial role in our analysis. We note that since $u^{N}$ and $w^{N}$ are piecewise constant in time, to establish monotonicity it is enough to verify that $u^{N}(\cdot, n\tmesh^{N})\leq w^{N}(\cdot, n\tmesh^{N})$ for all $n\in \mathbb{N}$. For the case $n=1$, note that by \eqref{e.approxscheme} applied to $w^N$, we have 
\[
w^N(x,\tmesh^N) = w^N(x,0) - \tmesh^N \mathcal{G}^N(w^N(x-\xmesh^N,0),w^N(x,0),w^N(x+\xmesh^N,0))\, .
\]
So to prove that $u^{N}(\cdot, \tmesh^{N})\leq w^{N}(\cdot, \tmesh^{N})$, it suffices to show that the map
\begin{equation}\label{e.nondec1}
(a,b,c) \mapsto b-\tmesh^{N}\mathcal{G}^{N}(a,b,c)
\end{equation}
is nondecreasing in all of its arguments, for all arguments which arise from the family~$\mathcal{C}$ (with $a=w^{N}(x-\xmesh^{N}, 0), b=w^{N}(x,t), c=w^{N}(x+\xmesh^{N}, 0)$). To extend beyond $n=1$, if we can further show that $u^{N}(\cdot, \tmesh^{N})$ and $w^{N}(\cdot, \tmesh^{N})$ both belong to $\mathcal{C}$, then by induction $u^{N}(\cdot, n\tmesh^{N})\leq w^{N}(\cdot, n\tmesh^{N})$ for all $n \in \N$, and thus $u^N \le w^N$ pointwise, so \eqref{e.gmon} holds. This is the general strategy by which we shall verify \eqref{e.gmon} throughout the rest of the paper. 

The result of Barles and Souganidis says that good approximation schemes for good initial value problems converge to the unique viscosity solutions of those initial value problems. In fact, their result applies to a broader family of PDEs and approximation schemes than the ones considered here; we have specialized their result, in order to make it easier to explain the application to our setting. 
\begin{thm}\label{t.bs2}[Theorem 1, \cite{BS}]
Fix a family $\mathcal{C}\subset L^{\infty}(\R)$, and a function $v_0 \in \mathcal{C}$. 
Consider a good initial value problem of the form \eqref{e.ivp2} started from initial condition $v_0$, 
and let $v$ be its unique viscosity solution. Next, 
fix an approximation scheme \eqref{e.approxscheme} which is a good approximation scheme for \eqref{e.ivp2} on $\mathcal{C}$, and for $N>0$ let $v^N$ be the solution of \eqref{e.approxscheme} with initial condition $v_0$. Then $v^N \to v$ locally uniformly as $N \to \infty$. 
\end{thm}

\subsection{A case-study: the central limit theorem for simple random walk.} \label{sec:casestudy}
In order to understand how Theorem~\ref{t.bs2} can be used to prove distributional limit theorems, we start with a simple example: in this section, we show how Theorem~\ref{t.bs2} can be applied to prove the Bernoulli central limit theorem.

Let $(S_n,n \ge 0)$ be a symmetric simple random walk on $\Z$, so $S_{n+1}=S_n\pm 1$, each with equal probability. 
Defining $p^n_k := \P(S_n=k)$ and $F^n_k := \P(S_n \le k)$, we then have 
\[
p^{n+1}_k = \frac{1}{2} p^n_{k-1}
			+\frac{1}{2} p^n_{k+1}. 
\]
Summing over $j \le k$ and rearranging gives 
\begin{equation}\label{e.srw_recurrence}
F^{n+1}_k - F^n_k =
			\frac{1}{2} (F^n_{k-1}
			-2 F^n_{k}
			+F^n_{k+1})\, .
\end{equation}

This recursion suggests that the appropriate PDE to describe the limiting behavior of the CDF is the heat equation $u_t=\tfrac{1}{2}u_{xx}$. With this in mind, we fix the family of functions $\mathcal{C}=\{f:\R\to[0,1]: f \mbox{ is a cumulative distribution function}\}$, fix $\eps > 0$ small, and let 
\begin{equation}\label{e.heateqintic}
v_0(x) = \frac{1}{(2\pi\eps)^{1/2}}\int_{-\infty}^x e^{-y^2/(2\eps)} \, dy.
\end{equation}

Consider the initial value problem 

\begin{equation}\label{eq:heatsmoothIC}
\begin{cases}
v_{t}-\frac{1}{2}v_{xx}=0&\text{in $\RR\times (0, \infty)$,}\\
v(x,0)=v_{0}(x)&\text{in $\RR$}\, .
\end{cases}
\end{equation}
Using the notation of \eqref{e.ivp2}, we have here $G(x,y) = -(1/2) y$. It is clear that $G$ is both continuous and degenerate elliptic. The initial condition, $v_0$, is an antiderivative of a Gaussian density, so it is both bounded and uniformly continuous and belongs to the class $\mathcal{C}$. Therefore, \eqref{eq:heatsmoothIC} is a \textit{good} initial value problem. 

Next we define the approximation scheme which arises in this setting.  We let $\xmesh^N = 1/N$ and $\tmesh^N = 1/N^2$, and set 
\begin{equation}\label{e.gndef}
\mathcal{G}^{N}(a,b,c) = -\frac{1}{\tmesh^N}\frac{1}{2}(a-2b+c).
\end{equation}
Then define $v^N$ via the scheme \eqref{e.approxscheme}, which we restate here for convenience:
\begin{equation}\label{e.approxschemesrw}
\begin{cases}
\frac{v^{N}(x,t+\tmesh^N)-v^{N}(x,t)}{\tmesh^N} +\mathcal{G}^{N}\langle v^{N}(x,t)\rangle_N=0, & \text{in $\R\times [\tmesh^{N}, \infty)$} \, ,\\
v^{N}(x,0)=v_{0}\big(\big\lfloor \frac{x}{\xmesh^N}\big\rfloor \xmesh^N\big), 
& \text{in $\R\times [0, \tmesh^N)$}\, .
\end{cases}
\end{equation}

 For $t \geq \tmesh^N$, this yields
\begin{align}\label{eq:SRWscheme}
\frac{v^N(x,t+N^{-2})-v^N(x,t)}{N^{-2}} &= 
\frac{v^N(x,t+\tmesh^N)-v^N(x,t)}{\tmesh^N} \notag \\
& = \frac{1}{\tmesh^N}\frac{1}{2}
(v^N(x-\xmesh^N,t)-2v^N(x,t)+v^N(x+\xmesh^N,t))\notag \\
& = \frac{1}{N^{-2}} \frac{1}{2} (v^N(x-1/N,t) - 2v^N(x,t)+v^N(x+1/N,t))\, .
\end{align}
This is simply a rescaling of the recurrence \eqref{e.srw_recurrence}. 
Thus, writing $\P_N$ for the measure under which $(S_n,n \ge 0)$ is a symmetric simple random walk with initial distribution given by 
\begin{equation}\label{e.smoothed_initialcondition}
\P_N(S_0 \le k) = v^N(k\xmesh^N,0)
=  \frac{1}{\sqrt{2\pi \eps}}\int_{-\infty}^{\frac{k}{N}}e^{-y^2/(2\eps)}dy\, ,
\end{equation}
then for all $n \in \N$ and $k \in \Z$ we have 
\begin{equation}\label{e.vfrelation}
\P_N(S_{n} \le k) = F^{n}_k = v^N(k/N,n/N^2)=v^N(k\xmesh^N,n\tmesh^N)\, .
\end{equation}

We now verify that with $\mathcal{G}^N$ as in \eqref{e.gndef}, the scheme \eqref{e.approxschemesrw} is a good approximation scheme of the PDE \eqref{eq:heatsmoothIC}. 
Fix initial conditions $u_0,w_0\in \mathcal{C}$ and let $u^N,w^N$ be defined by \eqref{e.approxschemesrw} with the initial condition $v_0$ replaced by $u_0,w_0$, respectively. 
Using that $\tmesh^N=N^{-2}$, we have that for any $a,b,c\in \R$, 
\[
b - \tmesh^{N}\mathcal{G}^N(a,b,c) = b+\frac{1}{2} (a-2b+c) = \frac{1}{2}(a+c)\, ,
\] 
so the map $(a,b,c) \mapsto b-\tmesh^N\mathcal{G}^N(a,b,c)$ is nondecreasing in all its arguments. It follows that $u^N(\cdot,\tmesh^N) \le w^N(\cdot,\tmesh^N)$. Moreover, by the analogues of \eqref{e.vfrelation} for $u^N$ and $w^N$, the functions $u^N(\cdot,\tmesh)$ and $w^N(\cdot,\tmesh)$ are again cumulative distribution functions, which lie in $\mathcal{C}$. It follows by induction that $u^N \le w^N$ everywhere; this establishes the monotonicity of the scheme. (Note that, in view of \eqref{e.vfrelation}, the monotonicity of the scheme is in fact {\em equivalent} to the statement that for simple random walk, if $S_0 \preceq S_0'$, then also $S_1 \preceq S_1'$. We could have used this to give a ``purely probabilistic'' proof of monotonicity, but we preferred to spell out the inductive approach to provide an example of its use in a simple setting. However, the connection between monotonicity and preservation of stochastic ordering under the random process dynamics will reoccur later in the paper.)

Next, the relation \eqref{e.vfrelation} implies in particular that the codomain of $v^N$ is contained within $[0,1]$, from which stability is immediate. 

Finally, if $\varphi:\R\times(0,\infty) \to \R$ is a bounded smooth function, then using that $\xmesh^N = 1/N$, we have 
\begin{align*}
\mathcal{G}^N\langle\varphi(y,s)+\eps\rangle_N
& = -\frac{1}{N^{-2}} 
\left(\frac{1}{2} (\varphi(y-1/N,s)+\eps - 2(\varphi(y,s)+\eps) + \varphi (y+1/N,s)+\eps )\right)\\
& =  
-\frac{1}{2} \frac{\varphi(y-1/N,s)-2\varphi(y,s)+\varphi(y+1/N,s)}{(1/N)^{2}}.
\end{align*}
This is a second centred difference approximation of $-\tfrac{1}{2}\varphi_{xx}(y,s)$, which converges to $-\tfrac{1}{2}\varphi_{xx}(x,t)$ as $(N,y,s) \to (\infty,x,t)$ in any order since $\varphi$ is smooth. (The $\eps\to0$ limit is irrelevant due to the linear role of $\eps$ in the approximation scheme.)

It follows that with $\mathcal{G}^{N}$ defined as in \eqref{e.gndef}, the approximation scheme defined by \eqref{e.approxschemesrw} is indeed a good approximation scheme for the corresponding initial value problem \eqref{eq:heatsmoothIC}. 
The (unique, continuous) viscosity solution of the heat equation \eqref{eq:heatsmoothIC} is 
\[
v(x,t) = \int_{-\infty}^x \frac{1}{\sqrt{2\pi(\eps+t)}} e^{-y^2/(2(\eps+t))}dy\, .
\]

Since $v^N(k/N,n/N^2)=\P_N(S_{n} \le k)$, taking $n=N^2$ and $k=xN$, and ignoring inconsequential rounding issues, we can now apply Theorem~\ref{t.bs2} to see that
\begin{equation}\label{eq:smoothed_limit}
\P_N(S_{N^2} \le xN) \to \int_{-\infty}^x \frac{1}{\sqrt{2\pi(\eps+1)}} e^{-y^2/(2(\eps+1))}dy =\P(\mathcal{W}_{\eps+1} \le x)\, ,
\end{equation}
where $\mathcal{W}_{\eps+1}$ is a centred Gaussian with variance $\eps +1$. 

The proof above required the use of a smooth initial condition $v_0$ in order to apply the result of Barles and Souganidis. However, probabilistically, we are interested in a simple random walk $S'_n$ which is started from 0. Because the Heaviside function cannot be thought of as the discretization of any Lipschitz initial condition, the result of Theorem \ref{t.bs2} is not applicable directly. We use the following coupling argument to connect a simple random walk $S_n$ started from a discretization of $v_0$, as defined above, and the simple random walk $S'_n$ started from 0. 

Fix $\delta > 0$, and let $(S_n',n \ge 0)$ be a simple random walk with initial condition $S_0'= 0$. The random walks $(S_n',n \ge 0)$ and $(S_n,n \ge 0)$ can be coupled so that both the increments of both walks are equal (i.e.\ so that $S_{n+1}'-S_n' = S_{n+1}-S_n$ for all $n \ge 0$). Under such a coupling, we have 
$S_{n}-S_{n}' = S_0-S_0' = S_0$
for all $n \ge 0$. 
It follows that if $S_{N^2}' \le xN$ then either $S_{N^2} \le (x+\delta)N$ or else $S_{0} > \delta N$, so 
\begin{align*}
\P(S_{N^2}' \le xN) 
	& \le \P_N(S_{N^2} \le (x+\delta) N) + \P_N(S_0 \ge \delta N)\, .
\end{align*}
By \eqref{e.smoothed_initialcondition} we have $\lim_{N \to \infty} \P_N(S_0 \ge \delta N) = \P(\mathcal{W}_\eps > \delta)$. Together with \eqref{eq:smoothed_limit} this yields that
\[
\limsup_{N \to \infty} \P(S_{N^2}' \le xN) 
\le \P(\mathcal{W}_{\eps+1} \le x+\delta) + \P(\mathcal{W}_{\eps} > \delta)\, .
\]
The identity $S_{n}-S_{n}' = S_0$ likewise implies that if $S_{N^2} \le (x-\delta)N$ then either $S_0 \le -\delta N$ or $S'_{N^2} \le xN$, so 
\begin{align*}
\P(S_{N^2}' \ge xN) 
	& \le \P_N(S_{N^2} \ge (x-\delta) N) - \P_N(S_0 \le -\delta N)\, ,
\end{align*}
	from which it similarly follows that 
\[
\liminf_{N \to \infty} \P(S_{N^2}' \le xN) 
\ge \P(\mathcal{W}_{\eps+1} \le x-\delta) - \P(\mathcal{W}_{\eps} < -\delta)\, .
\]
These bounds hold for all $\eps,\delta > 0$. (We'll later refer to this sort of argument as a sandwiching argument.) 

The first probabilities on the right in the preceding $\limsup\slash\liminf$ equations approximate $\P(\mathcal{W}_1 \le x)$ for $\eps,\delta$ small, and the second probabilities may be made as small as we like by choosing $\eps$ small as a function of $\delta$. We thus conclude that 
\[
\lim_{N \to \infty} \P(S_{N^2}' \le xN) = \P(\mathcal{W}_1 \le x), 
\]
where $\mathcal{W}_1$ is a centred Gaussian of variance 1. This is equivalent to the Bernoulli central limit theorem.

\subsection{The Barles-Souganidis framework and cooperative motion}\label{sec:bscm}

Moving from the heat equation $u_t -\tfrac12 u_{xx}=0$ to the porous medium equation $u_t -\tfrac12 (u^{m+1})_{xx}=0$, or its integrated version, the parabolic $p$-Laplace equation \eqref{eq:integrated}, makes the application of the Barles-Souganidis convergence framework more delicate. There are three issues, two related to monotonicity and one to the relevant PDE theory, which we now discuss.

First, the monotonicity required by Theorem~\ref{t.bs2} was obvious in the simple random walk/heat equation setting, and indeed held for the broadest natural class $\mathcal{C}$ of initial conditions: all CDFs. For cooperative motion, monotonicity simply does not hold in such generality, and to apply Theorem~\ref{t.bs2} we instead restrict our attention to the class $\mathcal{C}$ of Lipschitz CDFs. This only allows us to prove an approximation result (and hence convergence in distribution) for sequences of cooperative motion processes whose initial CDFs arise as discretizations of Lipschitz functions. 

The second issue relates to the use of (stochastic) monotonicity in the sandwiching argument. In Section~\ref{sec:casestudy} we used that simple random walks may be coupled so that if $S_0 \preceq S_0'$, then this stochastic ordering is maintained in time (by simply using the same increments for both walks). However, the dynamics of cooperative motion are {\em not} stochastically monotone for all initial distributions. That is, there exist initial conditions $X_0 \preceq X_0'$ for the cooperative motion process so that $X_1 \not\preceq X_1'$. 

This may seem like a death knell for the approach, given that stochastic monotonicity is required for the sandwiching argument. However, there is hope. The fact is that while stochastic monotonicity may not hold for {\em all} initial conditions of cooperative motion, it does hold for a large subclass of initial conditions. We will show that if, for example, the initial distributions $X_0 \preceq X_0'$ additionally satisfy that $\sup_k \P(X_0=k) < 1/e$ and $\sup_k \P(X'_0=k) < 1/e$, then indeed $X_1 \preceq X_1'$; the stochastic ordering is maintained. Moreover, under this condition, it turns out that $\sup_k \P(X_1=k) < 1/e$ and $\sup_k \P(X_1=k) < 1/e$, so the argument may be iterated in order to show that $X_n \preceq X_n'$ for all $n$. (Proving all this takes some work and requires a more technical coupling technique; this is accomplished in Section~\ref{sec:conv_pbounded}, below.) 

To extend the result to arbitrary initial distributions, we prove that from any initial condition, after a bounded number of steps, cooperative motion reaches a state where all single-site probabilities are small enough that stochastic monotonicity is obtained. (This fact feels unsurprising, and even seems like it should be ``obvious'' --  but the only proof we found is rather involved.) These first steps introduce a bounded error, which is washed away by the rescaling when we take limits. 

The final issue, which poses a significant challenge in comparison to the setting of the simple random walk, is to identify an explicit representation for the viscosity solution of the parabolic $p$-Laplace equation with suitable initial conditions (it is from this representation that we identify the limiting random variable in the distributional convergence). This was completely straightforward in the setting of the heat equation thanks to the theory of fundamental solutions for linear parabolic PDEs. In the case of the parabolic $p$-Laplace equation, no such theory exists. We overcome this obstacle using PDE techniques and analysis, which may be of independent interest. We introduce the main PDE ideas, and the explicit characterization of the relevant solution of the limiting PDE in Section~\ref{s:ZKB}.

\section{The ZKB solution and its approximation.}\label{s:ZKB}
This section presents the relevant PDE results needed to characterize the behaviour of the limiting dynamics. As briefly explained in the introduction, we seek an appropriate ``viscosity solution'' of the parabolic $p$-Laplace equation \eqref{e.normplaplace} started from a Heaviside initial condition. Our approach consists of building an approximation by continuous viscosity solutions in order to characterize this limit. This approximation is achieved by identifying the distributional solution of the porous medium equation \eqref{pme_entropy} with a suitable initial condition, known as the Zakharov-Kuznetsov-Burgers or {\em ZKB} solution, and considering its antiderivative.
\subsection{The ZKB solution.}\label{ss.zkb}
The evolution of the probability mass functions $(p^n_k, k \in \Z, n \ge 0)$ of the random variables $(X_n,n \ge 0)$ is a discrete approximation of the porous medium equation \eqref{pme_entropy}. If we assume that $u(x,0)=\delta_{0}(x)$ is a Dirac delta at zero, it turns out that the ``correct'' solution (in a sense to be specified in Section \ref{s.firstsol}) of \eqref{pme_entropy} corresponding to these dynamics is the ZKB solution (also called a source-type solution, source solution, Barenblatt solution, or Barenblatt-Pattle solution), which we now describe. Our presentation is based on that of \cite[Chapter 17]{vaz}. That reference presents the ZKB solution of 
$u_t -\Delta(u^{m+1})=0$ in $\R^{d}\times (0, \infty)$ for any dimension $d \ge 1$, but we  only present the solution for $u_{t}-\frac{1}{2}(u^{m+1})_{xx}=0$ in $\R\times (0, \infty)$, as this is all that is relevant for us. 

For every $\theta>0$, we define $\hat{U}(\cdot \,, \cdot \, ; \theta):\R\times (0, \infty)\rightarrow \RR$ by
\begin{equation}\label{e.Vrdef}
\hat{U}(x,t; \theta) = \frac{1}{t^{\frac{1}{m+2}}} \left[\left(\frac{\sqrt{2}\theta}{D\sqrt{m+1}}\right)^{\frac{2m}{m+2}}-\frac {2\rho|x|^2}{(m+1)t^{\frac{2}{m+2}}}\right]^{\frac{1}{m}}_+\, ,
\end{equation}
with $\rho:=m/(2(m+2))$ as defined in Theorem \ref{t.main} and $D:=2\int_{0}^{\infty}(1-\rho y^{2})_{+}^{\frac{1}{m}}\, dy$; we will see shortly this definition of $D$ also agrees with the one given in Theorem \ref{t.main}. As explained in \cite[Chapter 17, (17.31)]{vaz}, the choice of constants yields that for every $t>0$, 
\begin{equation}\label{eq:uhat_integral}
\int_{-\infty}^{\infty}\hat{U}(x,t; \theta)\, dx=\theta. 
\end{equation}
In particular, the function $\hat{U}(\cdot \,, \cdot \,; 1)$ has the explicit form 
\begin{equation}\label{e.Veardef}
\hat{U}(x,t; 1) = \frac{1}{t^{\frac{1}{m+2}}} \left[\left(\frac{\sqrt{2}}{D\sqrt{m+1}}\right)^{\frac{2m}{m+2}}-\frac {2\rho|x|^2}{(m+1)t^{\frac{2}{m+2}}}\right]^{\frac{1}{m}}_+\, ,
\end{equation}
which is supported on 
\[\left\{|x|\leq \frac{(m+1)^{\frac{1}{m+2}}}{2^{\frac{1}{m+2}}\rho^{\frac{1}{2}}D^{\frac{m}{m+2}}}t^{\frac{1}{m+2}}\right\}\, .
\]
Evaluating at $t=1$ yields 
\begin{equation*}
\hat{U}(x,1; 1) =\left[\left(\frac{\sqrt{2}}{D\sqrt{m+1}}\right)^{\frac{2m}{m+2}}-\frac {2\rho|x|^2}{(m+1)}\right]^{\frac{1}{m}}_+.
\end{equation*}

Up to an affine change of variables, $\hat{U}(x,1; 1)$ is a Beta density. Indeed, defining
\begin{align*}
f(x)&:= \hat{U}\left(x-(m+1)^{\frac{1}{m+2}}2^{-\frac{1}{m+2}}\rho^{-\frac{1}{2}}D^{-\frac{m}{m+2}},1;1\right)\\
&=\left[\left(\frac{\sqrt{2}}{D\sqrt{m+1}}\right)^{\frac{2m}{m+2}}-\frac {2\rho}{(m+1)}\left|x-\frac{(m+1)^{\frac{1}{m+2}}}{2^{\frac{1}{m+2}}\rho^{\frac{1}{2}}D^{\frac{m}{m+2}}}\right|^2\right]^{\frac{1}{m}}_+\\
&=\left(\frac{2\rho}{m+1}\right)^{\frac{1}{m}}\left(x\left(\frac{2\cdot (m+1)^{\frac{1}{m+2}}}{2^{\frac{1}{m+2}}\rho^{\frac{1}{2}}D^{\frac{m}{m+2}}}-x\right)\right)_{+}^{\frac{1}{m}}, 
\end{align*}
then $\int_{-\infty}^\infty f(x)dx=1$ by \eqref{eq:uhat_integral}. 
Next, if $X$ is distributed as 
\[\frac{2^{\frac{m+1}{m+2}}(m+1)^{\frac{1}{m+2}}}{D^{\frac{m}{m+2}}\rho^{\frac{1}{2}}}\Bet\left(\frac{m+1}{m}, \frac{m+1}{m}\right)\, ,
\] 
then the PDF of $X$ is given by 
\begin{align}\label{e.goalpdf}
& \frac{1}{\B\left(\frac{m+1}{m}, \frac{m+1}{m}\right)}\cdot \frac{1}{(2^{\frac{m+1}{m+2}}D^{-\frac{m}{m+2}}\rho^{-\frac{1}{2}}(m+1)^{\frac{1}{m+2}})^{\frac{m+2}{m}}}
\left(x\left(\frac{2\cdot(m+1)^{\frac{1}{m+2}}}{2^{\frac{1}{m+2}}\rho^{\frac{1}{2}}D^{\frac{m}{m+2}}}-x\right)\right)_{+}^{\frac{1}{m}}\, . 
\end{align}
Since this expression integrates to $1$, as does $f(x)$, and the two expressions are the same up to a multiplicative constant, they must in fact be equal. (The equality of the two expressions may also be verified directly, using that $D$ can be re-expressed as 
\begin{align*}
D=2\rho^{-\frac{1}{2}}\int_{0}^{1}(1-y^{2})^{\frac{1}{m}}\, dy=2\rho^{-\frac{1}{2}}\int_{0}^{1}(1-y)^{\frac{1}{m}}\frac{1}{2y^{\frac{1}{2}}}\, dy&=\rho^{-\frac{1}{2}}\B\left(\frac{1}{2}, \frac{m+1}{m}\right), 
\end{align*}
along with the identity 
\begin{equation*}
\Ga(2z)=(2\pi)^{-1/2}2^{2z-1/2}\Ga(z)\Ga\left(z+\frac{1}{2}\right)\, 
\end{equation*}
and the fact that $\B(x,y)=\tfrac{\Gamma (x)\,\Gamma (y)}{\Gamma (x+y)}$.) The fact that we fixed time $t=1$ above is arbitrary; a similar argument shows that $\hat{U}(x,t; 1)$ is a (scaled, shifted) Beta density for any $t>0$.

\subsection{The integrated ZKB solution, and the relation between distributional solutions and viscosity solutions}\label{s.firstsol}
For readers who are less familiar with the notions of distributional and viscosity solutions, we provide an overview in Appendix \ref{s.app}, which will be helpful to review before reading this section. 

The parabolic $p$-Laplace equation \eqref{e.normplaplace} is an integrated version of the porous medium equation~\eqref{pme_entropy}, and in our setting, the ``correct'' viscosity solution of \eqref{eq:integrated} and/or \eqref{e.normplaplace} is given by the antiderivative of the ZKB solution for the porous medium equation. This follows from a general result we prove relating distributional solutions and viscosity solutions in 1-dimension. 

Before beginning our analysis, we remark that we will always impose  hypotheses which guarantee existence and uniqueness for viscosity/distributional solutions. For the parabolic $p$-Laplace equation, as discussed near \eqref{e.ivp2}, the existence and uniqueness of viscosity solutions is guaranteed for any good IVP. For PDEs of the form 
\begin{equation}\label{eq:genericpde}
\begin{cases}
u_{t}-(\Psi(u))_{xx}=0&\text{in $\R\times (0, T]$}, \\
u(x,0)=u_{0}(x)&\text{in $\R$}, 
\end{cases}
\end{equation}
where $\Psi: \R\rightarrow \R$ is nondecreasing, continuous, and $\Psi(0)=0$, it follows from the main result of \cite{BreC}  that if $u_{0}\in L^{1}(\R)\cap L^{\infty}(\R)$ is nonnegative, then such IVPs have a unique distributional solution $u\in C([0, T]; L^{1}(\R))\cap L^{\infty}(\R\times [0, T])$. (See also \cite{BC} for related results.) Throughout the rest of the paper, by ``the unique distributional solution'' of a PDE of the form (\ref{eq:genericpde}), we always mean the one which belongs to $C([0, T]; L^{1}(\R))\cap L^{\infty}(\R \times [0, T])$. (In the special case of the porous medium equation with initial data $u_{0}\in L^{1}(\R)\cap L^{\infty}(\R)$, the existence and uniqueness of a distributional solution is also a consequence of the existence and uniqueness of strong $L^{1}$ solutions \cite[Theorem 9.3]{vaz}.)

We are now ready to present our main result concerning the connection between distributional and viscosity solutions. 
\begin{thm}\label{t.id}
Fix $\beta \geq 1$ and any nondecreasing function $\Psi \in C^{\beta}(\R)$ with $\Psi(0)=0$. 
Fix $u_{0}\in L^{1}(\RR)\cap L^{\infty}(\RR)$ and nonnegative, and for any $T>0$, let $u\in C([0, T]; L^{1}(\R))\cap L^{\infty}(\R\times [0, T])$ be the unique distributional solution of 
\begin{equation}\label{e.genpor}
\begin{cases}
u_{t}-(\Psi(u))_{xx}=0&\text{in $\R\times (0, T]$,}\\
u(x,0)=u_{0}(x)&\text{in $\R$.}
\end{cases}
\end{equation}
Let $v(x,t):=\int_{-\infty}^{x}u(y,t)\, dy$. Then $v$ is the unique viscosity solution of 
\begin{equation}\label{e.genp}
\begin{cases}
v_{t}-\Psi'(v_{x})v_{xx}=0&\text{in $\RR\times (0, T]$,}\\
v(x,0)=\int_{-\infty}^{x} u_{0}(y)\, dy&\text{in $\RR$.}
\end{cases}
\end{equation}
\end{thm}
A version of the correspondence described in Theorem~\ref{t.id} was heuristically presented in \cite{vaz} in the special case of the porous medium equation and the parabolic $p$-Laplace equation. The proof of Theorem~\ref{t.id} appears in Section~\ref{s.pde}, below.
\begin{remark}
We note that the uniqueness theory of both distributional and viscosity solutions yields that the converse of Theorem~\ref{t.id} essentially also holds. Indeed, if under the same hypotheses on $\Psi$, we work from the premise that $v$ is the unique viscosity solution of 
\begin{equation*}
\begin{cases}
v_{t}-\Psi'(v_{x})v_{xx}
=0
&\text{in $\RR\times (0, T]$,}\\
v(x,0)=v_{0}(x)&\text{in $\RR$,}
\end{cases}
\end{equation*}
with $v_{0}$ Lipschitz continuous, nonnegative, and bounded, then we may always express $v_{0}(x)=\int_{-\infty}^{x}v_{0}'(y)\, dy$, where $v_{0}'$ is defined pointwise a.e. by Rademacher's Theorem. Taking $u_{0}(x)=v_{0}'(x)$, then $u_{0}\in L^{1}(\RR)\cap L^{\infty}(\RR)$, and letting $u$ denote the unique distributional solution of \eqref{e.genpor} with this initial condition, Theorem \ref{t.id} and the uniqueness of viscosity solutions tells us {\em a posteriori} that $v(x,t)=\int_{-\infty}^{x}u(y,t)\, dy$. 
\end{remark}

\begin{remark}
Our result is similar in flavor to the well-known connection between entropy solutions of scalar conservation laws and viscosity solutions of Hamilton-Jacobi equations in 1-dimension (see for example \cite{cyril}). In fact, given our hypotheses on $\Psi$, it turns out that the unique distributional solution is in fact also the (unique) entropy solution of \eqref{e.genp}. This is consequence of \cite[Section 4, Corollary 9]{car}, which establishes (in a more general setting than ours) that distributional solutions are also entropy solutions. Therefore, Theorem \ref{t.id} can be viewed as a generalization of the classical result connecting entropy and viscosity solutions of scalar conservation laws/Hamilton-Jacobi equations. 
\end{remark}

As a special case of Theorem \ref{t.id}, we identify the unique viscosity solutions of the parabolic $p$-Laplace equation with bounded and Lipschitz continuous initial conditions as integrals of unique distributional solutions of the porous medium equation: 
\begin{cor}\label{c.id}
Fix nonnegative $u_{0}\in L^{1}(\R)\cap L^{\infty}(\R)$, and for any $T>0$, let $u\in C([0, T]; L^{1}(\R))\cap L^{\infty}(\R\times [0, T])$ denote the unique distributional solution of 
\begin{equation}\label{e.pmecor}
\begin{cases}
u_{t}-\frac{1}{2}(u^{m+1})_{xx}=0&\text{in $\R\times (0, T]$,}\\
u(x,0)=u_{0}(x)&\text{in $\R$.}
\end{cases}
\end{equation}
Then $v(x,t):=\int_{-\infty}^{x}u(y,t)\, dy$ is the unique viscosity solution of 
\begin{equation}\label{e.genp2}
\begin{cases}
v_{t}-\frac{m+1}{2}|v_{x}|^mv_{xx}
=0&\text{in $\RR\times (0, T]$,}\\
v(x,0)=\int_{-\infty}^{x} u_{0}(y)\, dy&\text{in $\RR$.}
\end{cases}
\end{equation}
\end{cor}

\begin{proof}
First, since $u_{0}$ is nonnegative, the comparison principle for the porous medium equation \cite[Theorem 9.2]{vaz} guarantees that $u$ is nonnegative on all of $\R\times [0, T]$, and hence we may rewrite \eqref{e.pmecor} as 
\begin{equation*}
\begin{cases}
u_{t}-\frac{1}{2}(|u|^{m}u)_{xx}=0&\text{in $\R\times (0, T]$,}\\
u(x,0)=u_{0}(x)&\text{in $\R$.}
\end{cases}
\end{equation*}
Letting $\Psi(r):=\frac{1}{2}|r|^{m}r$, then $\Psi$ satisfies the hypotheses of Theorem \ref{t.id}. Hence, by Theorem \ref{t.id}, $v(x,t):=\int_{-\infty}^{x}u(y,t)\, dy$ is the unique viscosity solution of \eqref{e.genp2}, as desired. 
\end{proof}

We highlight that Corollary \ref{c.id} has the requirement that $u_{0}\in L^{1}(\RR)\cap L^{\infty}(\RR)$, and in the case of the ZKB solution, this is simply not true (the ZKB solution at time 0 is a Dirac delta at the origin). We instead consider a ZKB solution shifted by some time $\eps>0$, which does indeed belong to $L^{1}(\RR)\cap L^{\infty}(\RR)$. Applying Corollary~\ref{c.id} with initial condition $\hat{U}(\cdot \,, \ve; r_{\ve})$ from \eqref{e.Vrdef} and with $r_{\ve}\rightarrow 1$ as $\eps \to 0$ identifies a family of viscosity solutions which approximate the desired solution of the parabolic $p$-Laplace equation with Heaviside initial conditions. 

This will be used as input to our sandwiching argument, comparing the CDF of the original $\scmmu$ process to the CDFs of these viscosity solution approximations. Picking suitable approximations and sending $\eps\to 0$, we will be able to conclude that the CDF is given precisely by the shifted Beta density \eqref{e.goalpdf}. The details of this argument appear in the proof of Proposition \ref{p.goodic}, below.

\subsection{The relationship between distributional solutions and viscosity solutions}\label{s.pde}

This section presents the proof of Theorem \ref{t.id}. We begin by recalling a result of Benilan and Crandall \cite{BC}, who considered equations of the form
\begin{equation}\label{eq:bc_form}
\begin{cases}
u_{t}-(\Psi(u))_{xx}=0&
\text{in $\R\times (0,T]$},\\
u(x,0)=u_{0}(x)&\text{in $\R$}, 
\end{cases}
\end{equation}
for $\Psi: \R\to \R$ continuous and nondecreasing, $u_0 \in L^1(\R) \cap L^\infty(\R)$, and for any $T>0$, with solutions considered in the distributional sense (see Definition \ref{d.ent}). In fact, \cite{BC} is more general than this; they present results for arbitrary dimensions, which we do not include here. 
The main result of \cite{BC} is the following stability property for solutions of equations of the form \eqref{eq:bc_form}.

\begin{thm}\cite[Theorem, page 162.]{BC}\label{t.bc2}
Let $(\Psi_n, 1\le n\le \infty)$ be nondecreasing, continuous functions $\Psi_n:\R \to \R$ with 
\begin{equation*}
\lim_{n\rightarrow \infty} \Psi_{n}(r)=\Psi_{\infty}(r)\quad\text{for all $r\in \R$},
\end{equation*}
and let $(u_{0n}, 1\leq n\leq \infty) \in L^1(\R)\cap L^\infty(\R)$ be such that 
\begin{equation}
\lim_{n\to \infty}\norm{u_{0n}-u_{0\infty}}_{L^1(\R)}=0.
\end{equation}
Fix $T > 0$ and for $1 \le n \le \infty$, let $u_n$ be the unique distributional solution of 
\begin{equation*}
\begin{cases}
(u_n)_{t}-\left( \Psi_{n}(u_{n})\right)_{xx}=0&
\text{in $\R\times (0,T]$},\\
u_{n}(x,0)=u_{0n}(x)&\text{in $\R$},
\end{cases}
\end{equation*}
with $u_n \in C([0,T]; L^{1}(\R))\cap L^{\infty}(\R\times [0,T])$. Then 
\begin{equation*}
u_{n}\rightarrow u_{\infty}\quad\text{in $C([0,T]; L^{1}(\R))$.}
\end{equation*}
In particular, this implies that 
\begin{equation*}
\lim_{n\rightarrow \infty} \sup_{t\in [0,T]} \int_{\R}|u_{n}(x,t)-u_\infty(x,t)|\, dx=0. 
\end{equation*}
\end{thm}

Equipped with Theorem \ref{t.bc2}, we proceed with the proof of Theorem \ref{t.id}. Our proof is a vanishing viscosity argument, where we add an additional viscosity term to regularize the PDE in order to obtain classical solutions, verify the relationship \eqref{e.anti} for the classical solutions, and use stability estimates to send the viscosity parameter to 0.

\begin{proof}[Proof of Theorem \ref{t.id}]
Fix $u_{0}\in L^{1}(\RR)\cap L^{\infty}(\RR)$ as in the statement of Theorem~\ref{t.id}, and let $u$ be the corresponding distributional solution of \eqref{e.genpor}. Since $u_{0}\in L^{1}(\RR)$, there exists a collection of functions $\left(u^{\ve}_{0}, 0< \ve\leq 1\right)\subseteq C^{\infty}(\RR)$ so that
\begin{equation}\label{e.L1}
\lim_{\ve\to0} \norm{u^{\ve}_{0}-u_{0}}_{L^{1}(\R)}=0.
\end{equation}
Since $\Psi\in C^{\beta}$ for some $\beta\geq 1$, we may further consider a sequence $(\tilde{\Psi}_{\ve}, 0<\ve<1)\subseteq C^{\infty}(\RR)$ so that 
\begin{align*}
&\tilde{\Psi}_{\ve}\xrightarrow[\ve\to0]{} \Psi\quad\text{locally uniformly, and}\\
&(\tilde{\Psi}_{\ve})'\xrightarrow[\ve\to0]{} \Psi'\quad{\text{locally uniformly}.}
\end{align*}
The construction of $\tilde{\Psi}_{\ve}$ can be done using a standard mollifier \cite[C.5]{evans}, and this construction additionally yields that $\tilde\Psi_{\ve}$ can be chosen to be nondecreasing for every $\ve>0$.

Now, for each $\eps > 0$, we define $\Psi_{\ve}(r):=\ve r+\tilde{\Psi}_{\ve}(r)$. We now consider the uniformly parabolic equation 
\begin{equation}\label{e.viscpme}
\begin{cases}
u^{\ve}_{t}-(\Psi_{\ve}(u^{\ve}))_{xx}=0&\text{in $\R\times (0, T],$}\\
u^{\ve}(x,0)=u^{\ve}_{0}(x)&\text{in $\R$}.
\end{cases}
\end{equation}
We may equivalently write the first line of \eqref{e.viscpme} as $u^{\ve}_{t}-\ve u^{\ve}_{xx}-\left(\tilde{\Psi}_{\ve}(u^\ve)\right)_{xx}=0$. 
By the introduction of the viscosity term $-\eps u^\eps_{xx}$, the problem \eqref{e.viscpme} becomes a uniformly parabolic quasilinear Cauchy problem, 
 with $\Psi_\eps$ smooth, and with principal part in divergence form. 
In particular, it is well known (see for example \cite[V, Theorem 8.1]{lady}) that 
under our hypotheses on $\Psi_{\ve}$ and $u^{\ve}_{0}$, 
for every $\ve>0$, $u^{\ve}\in C^{2,1}(\R\times (0,T])$, which implies (see Remark \ref{r.smooth}) that $u^{\ve}$ is a distributional solution of \eqref{e.viscpme}. Moreover, by \cite[V, Theorem 8.1]{lady} we have $u^{\ve}_{t}, u^{\ve}_{x}\in L^{\infty}(\R\times [0,T])$. 

Since $\Psi_{\ve}(r)\rightarrow \Psi(r) $ as $\eps\to 0$ for every $r\in \R$, and \eqref{e.L1} holds, it then follows by Theorem \ref{t.bc2} that for $u$ as in \eqref{e.genpor},
\begin{equation}\label{e.stable}
\lim_{\ve\rightarrow 0} \sup_{t\in [0,T]} \int_{\R}|u^{\ve}(x,t)-u(x,t)|\, dx=0. 
\end{equation}

We next define 
\begin{equation*}
v^{\ve}(x,t):=\int_{-\infty}^{x}u^{\ve}(y,t)\, dy. 
\end{equation*}
This makes $v^{\ve}\in C^{3,1}(\R\times (0, T])$ a classical solution (and therefore a viscosity solution) of 
\begin{equation*}
\begin{cases}
v^{\ve}_{t}-\left(\Psi_{\ve}(v^{\ve}_{x})\right)_{x}=0&\text{in $\R\times (0, T]$,}
\\
v^{\ve}(x,0)=\int_{-\infty}^{x}u^{\ve}(y,0)\, dy&\text{in $\R$,}
\end{cases}
\end{equation*}
which can be rewritten as 
\begin{equation*}
\begin{cases}
v^{\ve}_{t}-\Psi_{\ve}'\left(v^{\ve}_{x}\right)v^{\ve}_{xx}=0&\text{in $\R\times (0, T]$,}
\\
v^{\ve}(x,0)=\int_{-\infty}^{x}u^{\ve}(y,0)\, dy&\text{in $\R$.}
\end{cases}
\end{equation*}
Since $\Psi_{\ve}$ is nondecreasing, $\Psi'_{\ve}\geq 0$, and hence this problem is a good IVP, for which $v^{\ve}$ is the unique viscosity solution.

For $u$ the unique distributional solution of \eqref{e.genpor}, we let
\begin{equation*}
v(x,t):=\int_{-\infty}^{x} u(y,t)\, dy.
\end{equation*}
Equation \eqref{e.stable} implies that for any $K\subseteq \R$ compact, we have
\begin{align*}
\lim_{\ve\to 0}\sup_{x\in K}\sup_{t\in [0,T]}|v(x,t)-v^{\ve}(x,t)|&=\lim_{\ve\to 0}\sup_{x\in K}\sup_{t\in [0,T]}\left|\int_{-\infty}^{x}u(y,t)-u^{\ve}(y,t)\, dy\right|\\
&\leq \lim_{\ve\to 0} \sup_{x\in K}\sup_{t\in [0,T]}\int_{-\infty}^{x}\left|u(y,t)-u^{\ve}(y,t)\right|\, dy\\
&\leq \lim_{\ve\to 0} \sup_{t\in [0,T]}\int_{\RR}\left|u(y,t)-u^{\ve}(y,t)\right|\, dy=0,
\end{align*}
which implies that
\begin{equation*}
v^{\ve}\rightarrow v \quad\text{locally uniformly in $x,t$} \text{ on } \R \times [0,T] .
\end{equation*}
Since $\Psi_{\ve}'\rightarrow \Psi'$ locally uniformly, the stability property of viscosity solutions (see Proposition~\ref{p.stable}) now guarantees that $v$ is both a viscosity subsolution and supersolution of 
\begin{equation*}
v_{t}-\Psi'(v_{x})v_{xx}=0\quad\text{in $\R\times (0,T]$.}
\end{equation*}

Moreover, since $v\in C(\R\times [0, T])$, this yields that $v$ is the unique viscosity solution of 
\begin{equation*}
\begin{cases}
v_{t}-\Psi'(v_{x})v_{xx}=0&\text{in $\R\times (0,T]$,}\\
v(x,0)=\int_{-\infty}^{x}u(y,0)\, dy&\text{in $\R$},
\end{cases}
\end{equation*}
as desired.
\end{proof}

\begin{remark} While the proof of Theorem \ref{t.id} is a standard vanishing viscosity argument, we surprisingly were unable to find a version of Theorem \ref{t.id} anywhere in the literature, even for the porous medium equation. The idea to regularize PDEs of the form \eqref{eq:bc_form} and then study their stability properties was pursued by \cite{oleinik, sab}, who regularized the equations by considering initial data $u_{0}+\ve$ for $\ve>0, u_{0}\geq 0$.
\end{remark}

\section{Finite Difference Schemes for Diffuse Initial Conditions}\label{sect:LipIC}

In this section, we show that the CDFs of the discrete random variables $(X_n,n \ge 0)$ yield a good approximation scheme for the integrated ZKB solution, and that Theorem~\ref{t.bs2} applies in this setting, when the initial condition (the CDF of $X_0$) is given by a fine-mesh discretization of a Lipschitz continuous extended CDF.

Throughout the section, fix $m\in (0, \infty)$, and a probability distribution $\mu$ supported on $\Z\cup\{-\infty,\infty\}$. Let $(X_n,n \ge 0)$ be $\scmmu$-distributed, and for $k \in \Z$, write $F^n_k=F^n_k(\mu)=\P(X_n \leq k)=\mu[-\infty,k]$. In this case $F^{n}_{(\cdot)}$ is an extended CDF for each $n\in \N$. We suppress the dependence on $m$ as it is fixed throughout, and also suppress the dependence on $\mu$ whenever possible. As observed in Section \ref{s.intro}, $(F^n_k)_{k \in \Z, n\in \N}$ can be defined by the recurrence
\begin{equation}\label{e.plap1}
\begin{cases}
F^{n+1}_k  = F^n_k+\frac{1}{2}\left[(F^n_{k+1} - F^n_{k})^{m+1} - (F^n_k - F^n_{k-1})^{m+1}\right],\\
F^{0}_{k}=\mu[-\infty, k]. 
\end{cases}
\end{equation}

As before, for every $N>0$, we consider mesh sizes $\xmesh^{N}$ and $\tmesh^{N}$ such that $\xmesh^{N}, \tmesh^{N}\to 0$ as $N\to \infty$. Due to the scaling properties of the parabolic $p$-Laplace equation, the relation $\tmesh^{N}=(\xmesh^{N})^{m+2}$ is natural (when $m=0$ this is Brownian scaling), and we enforce this relation on $\tmesh^N$ and $\xmesh^N$ throughout, by defining 
\begin{equation}\label{e.mesh}
\tmesh^{N}=\frac{1}{N}\quad\mbox{ and }\quad\xmesh^{N}=\frac{1}{N^{1/(m+2)}}. 
\end{equation}
For brevity, we suppress the dependence on $N$ wherever possible. Given $v_0 \in L^\infty(\R)$, we define a function $v^N \in L^\infty(\R\times[0,\infty))$ by 
\begin{equation}\label{e.plap2}
\begin{cases}
v^{N}(x,t+\tmesh)=v^{N}(x,t)+ \frac{\tmesh}{\xmesh}\frac{1}{2}\left[\left(\frac{v^{N}(x+\xmesh, t)-v^{N}(x,t)}{\xmesh}\right)^{m+1}\right.& \\
\quad\qquad \qquad \qquad \qquad\qquad\qquad -\left.\left(\frac{v^N(x,t)-v^N(x-\xmesh,t)}{\xmesh}\right)^{m+1}\right]
&\text{in $\R\times [\tmesh, \infty)$} \, ,
\\
v^{N}(x,t)=
v_{0}\big(\big\lfloor \frac{x}{\xmesh}\big\rfloor \xmesh\big)
&\text{in $\R\times [0,\tmesh)$} \, .
 \end{cases}
\end{equation}
\begin{remark}\label{r.vdef}
If $v_0$ is an extended CDF, then for $N \in \N$ we may define a probability measure $\mu^N$ on $\Z\cup\{-\infty,\infty\}$ by setting $\mu^N[-\infty,k] = v_0(k\xmesh^N)=v_0(k/N^{1/(m+2)})$. With this definition, if $(X_n,n \ge 0)$ is $\mathrm{SCM}(m,\mu^N)$-distributed then for $k \in \Z$, 
\[
v^N(k\xmesh^N,0) = v_0(k\xmesh^N) = F^0_k\, ,
\]
and inductively, $v^N(k\xmesh^N,n\tmesh^N) = F^n_k(\mu^N)$ for all $n \in \N$ and $k \in \Z$. Since $v^N$ is piecewise constant in each lattice rectangle of the mesh $\xmesh \Z\times \tmesh \N=N^{-1/(m+2)}\Z\times N^{-1} \N$, interpreted with closed lower and left boundaries and open upper and right boundaries, for all $(x,t) \in \R\times[0,\infty)$, we may then express $v^N(x,t)$ as 
\begin{equation}\label{e.vdef}
v^{N}(x,t)=F^{\lfloor Nt \rfloor}_{\lfloor N^{1/(m+2)}x \rfloor}\, ,
\end{equation}
where $(F^n_k)_{k \in \Z, n\in \N}$ is defined by \eqref{e.plap1}. 
In particular, this implies that for any $t > 0$, $v^N(\cdot,t)$ is an extended CDF which is piecewise constant on intervals of the form $[kN^{-1/(m+2)},(k+1)N^{-1/(m+2)})$, for $k \in \Z$. This also yields that in this case $v^N \in L^\infty(\R\times[0,\infty))$.
\end{remark}

 The main result of this section shows that, under suitable assumptions on the initial conditions, $v^{N}\rightarrow v$ locally uniformly, where $v$ is the unique viscosity solution of the parabolic $p
 $-Laplace equation. 

\begin{prop}\label{p.intlip}
Let $v_0$ be a Lipschitz continuous extended CDF with Lipschitz constant $M$. Fix $N \in \N$ and define a probability distribution $\mu_N$ on $\Z \cup \{-\infty,\infty\}$ by  
\[
\mu_N[-\infty,k]:=v_0(kN^{-1/(m+2)}) 
\]
for $k \in \Z$. Let $(X_n,n \ge 0)$ be $\mathrm{SCM}(m, \mu_{N})$-distributed, and let $F^n_k = F^n_k(\mu_N)$ be defined by \eqref{e.plap1}.
Finally, fix $T>0$, and $K\subseteq \R$ compact. Then for every $\ve>0$, there exists $N_0=N_0(m, M, \ve, K, T)$ such that if $N \ge N_0$, then for $v^N$ defined by \eqref{e.plap2}, 
\begin{equation}\label{e.invp}
\sup_{0 \le t \le T} 
\sup_{x\in K} 
\left|
v^{N}(x,t) - v(x,t)
\right|\leq \ve ,
\end{equation}
where $v$ is the viscosity solution of 
\begin{equation}\label{e.plap}
\begin{cases}
v_{t}-\frac{m+1}{2}|v_{x}|^{m+1}v_{xx}=0&\text{in $\RR\times (0, T]$,}\\
v(x,0)=v_{0}(x)&\text{in $\RR$.}
\end{cases}
\end{equation}
It follows that $v(x,1)$ is an extended CDF and that if $N\geq N_{0}$, then 
\begin{equation}\label{e.invp1}
\sup_{x \in K} 
\left|
\P\left\{\frac{X_N}{N^{1/(m+2)}}\leq x \right\}-v(x,1)
\right| \leq \ve.
\end{equation}
\end{prop}

In order to prove Proposition \ref{p.intlip}, we must verify that the conditions of Theorem \ref{t.bs2} are satisfied. 
We begin with a monotonicity lemma for certain solutions of \eqref{e.plap1}. Write $p^{*}:=\frac{1}{(m+1)^{1/m}}$. We say that a probability distribution $\mu$ on $\Z$ is $p^*$-bounded if $\mu(\{z\}) \le p^*$ for all $z \in \Z$.
\begin{lemma}\label{lem:p*_monotonicity}
Fix probability distributions $\mu,\nu$ on $\Z\cup\left\{-\infty, \infty\right\}$. 
Let $(G^n_k)_{k \in \Z,n \in \N}=(F^n_k(\mu))_{k \in \Z,n \in \N}$ and $(H^n_k)_{k \in \Z,n \in \N}=(F^n_k(\nu))_{k \in \Z,n \in \N}$ be defined by the recurrence \eqref{e.plap1} with initial conditions given by $\mu$ and $\nu$, respectively. If $\mu$ and $\nu$ are $p^*$-bounded 
and $\mu[-\infty,k] \le \nu[-\infty,k]$ for all $k \in \Z$, 
then $G^n_k \le H^n_k$ for all $k \in \Z$ and $n \in \N$, and additionally $G^n_k-G^n_{k-1} \le p^*$ and $H^n_k-H^n_{k-1} \le p^*$ for all $k \in \Z$ and $n \in \N$. 
\end{lemma}
\begin{proof}
Write
\[
\mathcal{S}(a,b,c) 
= b + \frac{1}{2}[(c-b)^{m+1}-(b-a)^{m+1}]\, ,
\]
so that $F^{n+1}_k = \mathcal{S}(F^n_{k-1},F^n_k,F^n_{k+1})$.
Note that $\mathcal{S}$ is nondecreasing in $a$ when $a \le b$ and is nondecreasing in $c$ when $b \le c$. Moreover, 
\[
\frac{\partial}{\partial b} \mathcal{S}(a,b,c) 
= 1 - \frac{m+1}{2}[(c-b)^m +(b-a)^m]
\]
which is nonnegative provided that $0 \le b-a\le p^*$ and $0\le c-b \le p^*$. This in particular shows that the directional derivative of $\mathcal{S}$ in directions $(1,0,0),(0,1,0)$ and $(0,0,1)$ is positive on the set 
\[
R=\big\{(a,b,c): a \ge 0 , b-a \in [0, p^*], c-b \in [0, p^*]\big\}.
\]
It follows that if $(a,b,c),(a',b',c')\in R$ and $a \le a'$, $b \le b'$, $c \le c'$, then the directional derivative of $\mathcal{S}$ is also nonnegative on the set $R$ in direction $(x,y,z):=(a'-a,b'-b,c'-c)$, since this direction is a linear combination of the axis directions. Since $R$ is convex, the line segment from $(a,b,c)$ to $(a',b',c')$ lies within $\mathcal{S}$, and we conclude that $\mathcal{S}(a,b,c) \le \mathcal{S}(a',b',c')$.

Recalling that $G^0_k=\mu[-\infty,k]$ and $H^0_k=\nu[-\infty,k]$, under the assumptions of the lemma 
 it follows that for all $k \in \Z$, we have the entry-wise inequality
$(G^0_{k-1},G^0_k,G^0_{k+1}) \le (H^0_{k-1},H^0_k,H^0_{k+1})$, and both triples lie in $R$, so 
\[
G^1_k = \mathcal{S}(G^0_{k-1},G^0_k,G^0_{k+1})
\le \mathcal{S}(H^0_{k-1},H^0_k,H^0_{k+1}) = H^1_k\, .
\]

We now claim that $0 \le G^1_k-G^1_{k-1} \le p^*$ and $0 \le H^1_k-H^1_{k-1} \le p^*$ for all $k \in \Z$, in which case it follows by induction that $G^n_k \le H^n_k$ and that $G^n_k-G^n_{k-1} \le p^*$ and $H^n_k-H^n_{k-1}\le p^*$ for all $n \in \N$ and $k \in \Z$. We only prove that $0 \le G^1_k-G^1_{k-1}\le p^*$ since an identical argument works for $H$. To show this, note that for any $\lambda \in \R$, 
\[
\mathcal{S}(a+\lambda,b+\lambda,c+\lambda)
=b+\lambda + \frac12[(c-b)^{m+1}-(b-a)^{m+1}]
= \lambda+\mathcal{S}(a,b,c)\, .
\]
Since $G^0_{k-1} \le G^0_{k} \le G^0_{k-1}+p^*$ for all $k \in \Z$, it follows that 
\[
(G^0_{k-2},G^0_{k-1},G^0_k)
\le 
(G^0_{k-1},G^0_k,G^0_{k+1})
\le 
(G^0_{k-2}+p^*,G^0_{k-1}+p^*,G^0_k+p^*)
\]
Since $\mathcal{S}$ is nondecreasing in all its arguments on $R$ and all the above triples lie in $R$, it follows that 
\[
G^1_{k-1} = \mathcal{S}(G^0_{k-2},G^0_{k-1},G^0_k)
\le 
\mathcal{S}(G^0_{k-1},G^0_k,G^0_{k+1})
= G^1_k\, ,
\]
and
\begin{align*}
G^1_k & =\mathcal{S}(G^0_{k-1},G^0_k,G^0_{k+1})\\
& \le \mathcal{S}(G^0_{k-2}+p^*,G^0_{k-1}+p^*,G^0_k+p^*) \\ 
& = p^*+\mathcal{S}(G^0_{k-2},G^0_{k-1},G^0_k) \\
& = p^*+G^1_{k-1}\, ,
\end{align*}
as required.
\end{proof}
As a straightforward consequence of the lemma, we obtain that the monotonicity condition \eqref{e.gmon} holds whenever the functions are ``generated'' by Lipschitz continuous extended CDFs. For $M > 0$ write 
\[
\mathcal{C}_M:=\left\{w\in L^{\infty}(\R): \text{$w$ is an $M$-Lipschitz continuous extended CDF}\right\}.
\]
The following corollary now verifies that the approximation scheme~\eqref{e.plap2} is monotone on $\mathcal{C}_M$. 
\begin{cor}\label{l.smoothmon}
For all $M > 0$ there exists $N_0$ such that for $N \ge N_0$ the following holds. Fix $u_0,w_0 \in \mathcal{C}_M$ and let $u^N$ and $w^N$ be defined by \eqref{e.plap2} with initial conditions given by $u_0$ and $w_0$, respectively. 
If $u_0 \le w_0$, then $u^N \le w^N$. 
\end{cor}
\begin{proof}
Define probability distributions $\mu^N$ and $\nu^N$ on $\Z\cup \{-\infty,\infty\}$ by setting 
$\mu^N[-\infty,k]=u_0(kN^{-1/(m+2)})$ and $\nu^N[-\infty,k]=w_0(kN^{-1/(m+2)})$. 
By the observation of Remark~\ref{r.vdef}, we then have 
$u^{N}(x,t):=F^{\lfloor Nt \rfloor}_{\lfloor N^{1/(m+2)}x \rfloor}(\mu^N)$ and $w^{N}(x,t):=F^{\lfloor Nt \rfloor}_{\lfloor N^{1/(m+2)}x \rfloor}(\nu^N)$. 
Moreover, if $N \ge N_0 := M^{m+2} (m+1)^{(m+2)/m}$ then for all $k \in \Z$ we have 
\[
\mu^N(\{k\})=u_0(kN^{-1/(m+2)})
-u_0((k-1)N^{-1/(m+2)}) \le M \frac{1}{N^{1/(m+2)}} \le p^*\, ,
\] 
so by Lemma~\ref{lem:p*_monotonicity}, it follows that for $N \geq N_0$, $u^{N}(x,t) \le w^N(x,t)$ for all $x \in \R$ and $t \in [0,\infty)$, as required. 
\end{proof}

Having verified monotonicity, we are now ready to use Theorem \ref{t.bs2} to prove Proposition \ref{p.intlip}.

\begin{proof}[Proof of Proposition~\ref{p.intlip}]
Define $G:\R\times\R \to \R$ by
\[
G(v_x,v_{xx}) := -\frac{m+1}{2} |v_x|^mv_{xx}\, .
\]
For $a,b \in \R$ with $a \le b$ we then have $G(p,a) \ge G(p,b)$ for any $p \in \R$, i.e., $G$ is degenerate elliptic. 
Since $G$ is also continuous, it follows that if $v_0$ is any Lipschitz continuous extended CDF, then the initial value problem 
\begin{equation}\label{eq:modifiedivp}
\begin{cases}
v_{t}+G(v_x,v_{xx})=0&\text{in $\RR\times (0, \infty)$,}\\
v(x,0)=v_{0}(x)&\text{in $\RR$}\, 
\end{cases}
\end{equation}
is good. 

Writing $\xmesh=\xmesh^N$ and $\tmesh=\tmesh^N$, define 
\begin{align*}
\mathcal{G}^N(a,b,c)
& = -\frac{1}{\xmesh}\frac12
\left[
\left|
\frac{c-b}{\xmesh}
\right|^{m}\left(
\frac{c-b}{\xmesh}
\right)
-
\left|
\frac{b-a}{\xmesh}
\right|^{m}
\left(
\frac{b-a}{\xmesh}
\right)
\right] \\
& = 
-\frac{1}{2\tmesh} 
\left[
|c-b|^{m}(c-b)-|b-a|^{m}(b-a)
\right]
\, ,
\end{align*}
where we have used that $\xmesh^{m+2}=\tmesh$ for the second equality. 
Recalling the notation $\langle v^N(x,t)\rangle_N = (v^N(x,t-\tmesh),v^N(x,t),v^N(x,t+\tmesh))$, if $v^N$ is defined by \eqref{e.plap2}, then since $v^N$ is nondecreasing in $x$ for all $t$, we  have 
\[
v^N(x,t+\tmesh) = v^N(x,t)-\tmesh\mathcal{G}^N\langle v^N(x,t)\rangle_N\, ,
\]
so we may rewrite \eqref{e.plap2} in the form given by \eqref{e.approxscheme}: 
\begin{equation} \label{e.plap2rewrite}
\begin{cases}
\frac{v^{N}(x,t+\tmesh)-v^{N}(x,t)}{\tmesh} +\mathcal{G}^{N}\langle v^{N}(x,t)\rangle_N=0, & \text{in $\R\times [\tmesh, \infty)$}\, ,\\
v^{N}(x,0)=v_{0}\big(\big\lfloor \frac{x}{\xmesh}\big\rfloor \xmesh\big), 
& \text{in $\R\times [0, \tmesh)$}\, .
\end{cases}
\end{equation}
Since $v_0$ is a Lipschitz continuous extended CDF, Corollary~\ref{l.smoothmon} then ensures that the monotonicity condition~\eqref{e.gmon} is satisfied for $N$ sufficiently large. Moreover, for all $t$, $v^N(\cdot,t)$ is a piecewise constant approximation of an extended CDF, so 
$\sup_{x \in \R,t \ge 0} |v^N(x,t)| \le 1$, which verifies stability. 
To check consistency, note that for $\varphi:\R \times(0,\infty) \to \R$ and $\eps > 0$ we have 
\begin{align*}
\mathcal{G}^{N}\langle \vp(y,s)+\eps\rangle_{N}&=-\frac{1}{\xmesh}\frac{1}{2}\left[\left|\frac{\vp(y+\xmesh,s)-\vp(y,s)}{\xmesh}\right|^{m}\frac{\vp(y+\xmesh,s)-\vp(y,s)}{\xmesh}\right.\\
& \quad\quad\quad\quad
  -\left.\left|\frac{\vp(y,s)-\vp(y-\xmesh,s)}{\xmesh}\right|^{m}\frac{\vp(y)-\vp(y-\xmesh,s)}{\xmesh}\right], \\
& = \mathcal{G}^{N}\langle \vp(y,s)\rangle_{N}\, ,
\end{align*}
so if $\vp$ is smooth and bounded then 
\begin{align*}
\lim_{(N, y,s,\eps)\to (\infty, x_{0}, t_{0},0)} \mathcal{G}^{N}\langle \vp(y,s)+\eps \rangle_{N} & = 
\lim_{(N, y,s)\to (\infty, x_{0}, t_{0})} \mathcal{G}^{N}\langle \vp(y,s)\rangle_{N}\\
&=-\frac{1}{2}((\vp_{x}(x_{0}, t_{0}))^{m+1})_{x}\\
& = -\frac{m+1}{2}\vp_x(x_0,t_0)\vp_{xx}(x_0,t_0)\\
&=G(\vp_{x}(x_{0}, t_{0}), \vp_{xx}(x_{0}, t_{0})),\\
&= -\frac{1}{2}(|\vp_{x}(x_{0}, t_{0})|^{m}\vp_x(x_0,t_0))_{x}\\
& = -\frac{m+1}{2}|\vp_x(x_0,t_0)|^m\vp_{xx}(x_0,t_0)\\
&= G(\vp_{x}(x_{0}, t_{0}), \vp_{xx}(x_{0}, t_{0}))
\end{align*}
and this establishes consistency. 

We have just verified that \eqref{e.plap2}, or equivalently \eqref{e.plap2rewrite}, is a good approximation scheme for \eqref{eq:modifiedivp}. 
It follows by Theorem \ref{t.bs2} that for every $T>0$, for every $K\subseteq \R$ compact, and for $N\geq N_{0}$ sufficiently large, we have 
\begin{equation*}
\sup_{0\leq t\leq T} \sup_{x\in K} |v^{N}(x,t)-v(x,t)|\leq \ve, 
\end{equation*}
where $v$ is the viscosity solution of 
\begin{equation}\label{e.plap*}
\begin{cases}
v_{t}-\frac{m+1}{2}|v_{x}|^{m}v_{xx}=0&\text{in $\RR\times (0, T]$,}\\
v(x,0)=v_{0}(x)&\text{in $\RR$.}
\end{cases}
\end{equation}

For the second assertion of Proposition \ref{p.intlip}, we note that since
\begin{equation*}
v^{N}(x,t)=F^{\lfloor Nt \rfloor}_{\lfloor N^{1/(m+2)}x \rfloor}=\p{X_{\lfloor Nt \rfloor}\leq \lfloor N^{1/(m+2)}x \rfloor}=\p{X_{\lfloor Nt \rfloor}\leq N^{1/(m+2)}x}, 
\end{equation*}
we have that for $t=1$,  
\begin{align*}
 \sup_{x\in K} \left|\p{\frac{X_{N}}{N^{1/(m+2)}}\leq x}-v(x,1)\right|\leq \ve.
\end{align*}
Finally, since $v(\cdot,1)$ is nondecreasing and continuous and is the pointwise limit of extended CDFs, it follows that $v(\cdot,1)$ itself is an extended CDF. 
\end{proof}

\section{Convergence under $p^{*}$-bounded initial conditions}\label{sec:conv_pbounded}

Recall from Section~\ref{sect:LipIC} that $p^* = \tfrac{1}{(m+1)^{1/m}}$ and that we say a probability distribution $\mu$ on $\Z$ is $p^*$-bounded if $\mu(\{z\}) \le p^*$ for all $z \in \Z$.
The main result of this section is that $\scmmu$-distributed cooperative motions converge in distribution to a scaled and shifted Beta random variable when $\mu$ is $p^*$-bounded. 

\begin{prop}\label{p.goodic}
Let ($X_n$, $n \geq 0$) be \scmmu-distributed with $\mu$ a probability distribution on $\Z$. If $\mu$ is $p^*$-bounded, then 
$$\lim_{n\to\infty} \P\left(\frac{X_n}{n^{1/(m+2)}} \leq x\right) = v(x, 1)$$ locally uniformly in $x$, where $v(x,1)$ is the CDF of 
\begin{equation}\label{e.genbeta}
\frac{2^{\frac{m+1}{m+2}}(m+1)^{1/(m+2)}}{D^{\frac{m}{m+2}}\rho^{1/2}}\left[\Bet\left(\frac{m+1}{m}, \frac{m+1}{m}\right)-\frac{1}{2}\right].
\end{equation}
\end{prop}

\begin{proof}[Proof of Proposition \ref{p.goodic}]
Fix an $\varepsilon > 0$ and let $(X_n, n\geq 0)$ be \scmmu-distributed. Let $F^0_k := \P(X_0 \leq k) = \mu(-\infty,k]$. Using \eqref{e.plap2} with $v_{0}(x)=F^{0}_{\lfloor N^{1/(m+2)}x\rfloor}$, we have that $v^N(x,t)= F^{\lfloor Nt \rfloor}_{\lfloor N^{1/(m+2)} x\rfloor}$. 
We will choose schemes $v^{-,N}$ and $v^{+,N}$ with smoother initial conditions which are $p^{*}$-bounded and which sandwich $v_{0}$ (i.e. bound it from below and above). Proposition \ref{p.intlip} will allow us to conclude that $v^{-, N}$ and $v^{+,N}$ converge as $N\to\infty$, while Lemma \ref{lem:p*_monotonicity} will guarantee that $v^{-,N}$ and $v^{+,N}$ sandwich $v^N$ for all times. 

By \eqref{e.Vrdef}, we have 
\begin{equation}\label{e.Vdef}
\hat{U}(x,t; 1-\ve) = \frac{1}{t^{\frac{1}{m+2}}} \left[\left(\frac{\sqrt{2}(1-\ve)}{D\sqrt{m+1}}\right)^{\frac{2m}{m+2}}-\frac {2\rho|x|^2}{(m+1)t^{\frac{2}{m+2}}}\right]^{1/m}_+\, ,
\end{equation}
We then define
\begin{equation}\label{eq:vvar}
v_\varepsilon(x,t) := \int_{-\infty}^x \hat{U}(y,t+\varepsilon;1-\varepsilon) \, dy.
\end{equation}

Let 
\begin{equation}\label{eq:Pi}
\Pi^{1-\ve}(\ve):=\frac{(1-\ve)^{\frac{m}{m+2}}(m+1)^{\frac{1}{m+2}}}{2^{\frac{1}{m+2}}\rho^{\frac{1}{2}}D^{\frac{m}{m+2}}}\ve^{\frac{1}{m+2}}, 
\end{equation}
so that the support of $\hat{U}(\cdot, \ve; 1-\ve)$ is given by $[-\Pi^{1-\ve}(\ve), \Pi^{1-\ve}(\ve)]$. 
(A more general function $\Pi^\theta(t)$ is used in the Appendix, and this notation is chosen to agree with that notation.)
In particular, this implies that $v_\varepsilon(x,0)$ is constant on $(-\infty, -\Pi^{1-\ve}(\ve))$ and on $(\Pi^{1-\ve}(\ve),\infty)$. Since $\hat{U}(\cdot, \ve; 1-\ve)$ belongs to $L^{\infty}(\R)$, $v_{\ve}(x,0)$ is Lipschitz continuous. Note that $v_\varepsilon(\Pi^{1-\ve}(\ve),0) = 1-\varepsilon$, by the choice of constant in the definition of $\hat{U}(\cdot \,, \cdot \,; 1-\ve)$. 

Define 
\begin{align*}
L&:= L(\varepsilon) = \max\{k\leq 0: F^0_k \leq \varepsilon\}\\
R&:= R(\varepsilon) = \min\{k \geq 0: F^0_k \geq 1-\varepsilon\}.
\end{align*}

These values are both finite because $\mu$ is a probability distribution on $\R$. Then for $\tilde{n} \in \N$, let $v^{+,N}(x,t;\tilde{n})$ be defined by \eqref{e.plap2} with $(\xmesh)^{m+2} = \tmesh=1/N$ and initial condition
\begin{equation}\label{e.v+def}
v^{+,N}(x,t;\tilde{n}) = v_\varepsilon(\lfloor (x - L\tilde{n}^{-1/(m+2)} + 2\Pi^{1-\ve}(\ve))/\xmesh \rfloor \xmesh,0) + \varepsilon
\end{equation}
for $t \in [0,\tmesh)$.

Similarly, define $v^{-,N}(x,t;\tilde{n})$ by \eqref{e.plap2} with $(\xmesh)^{m+2} = \tmesh=1/N$ and initial condition
\begin{equation}\label{e.v-def}
v^{-,N}(x,t;\tilde{n}) = v_\varepsilon(\lfloor (x- R\tilde{n}^{-1/(m+2)}-\Pi^{1-\ve}(\ve))/\xmesh \rfloor\xmesh,0)
\end{equation}
for $t \in [0,\tmesh)$.

Since the initial condition is bounded and uniformly continuous, by Theorem \ref{t.bs2} we know that $v^{+,N}(x,t;\tilde{n})$ converges to the unique viscosity solution $v^{+}(x,t;\tilde{n})$ of
\begin{equation}\label{eq:upperPDE}
\begin{cases}
v^{+}_{t}-\frac{m+1}{2}\left|v^{+}_{x}\right|^{m}v^+_{xx}=0&\text{in $\RR\times (0,T]$,}\\
v^{+}(x,0)=v_{\varepsilon}(x-L\tilde{n}^{-1/(m+2)}+2\Pi^{1-\ve}(\ve),0) + \varepsilon &\text{in $\RR$.}
\end{cases}
\end{equation}

Similarly, $v^{-,N}(x,t;\tilde{n})$ converges to the unique viscosity solution $v^{-}(x,t;\tilde{n})$ of
\begin{equation}\label{eq:lowerPDE}
\begin{cases}
v^{-}_{t}-\frac{m+1}{2}\left|v^{-}_{x}\right|^{m}v^-_{xx}=0&\text{in $\RR\times (0, T]$,}\\
v^{-}(x,0)=v_{\varepsilon}(x-R\tilde{n}^{-1/(m+2)}-\Pi^{1-\ve}(\ve),0)&\text{in $\RR$.}
\end{cases}
\end{equation}
Notice that the initial conditions in \eqref{eq:upperPDE} and \eqref{eq:lowerPDE} are simply shifts of $v_\varepsilon(x,0)$. Therefore, by Corollary~\ref{c.id} and Lemma \ref{l.barenblatt}, we know the profiles of $v^+(x,t;\tilde{n})$ and $v^-(x,t;\tilde{n})$ are the corresponding shifts of $\int_{-\infty}^x \hat{U}(y,t+\varepsilon;1-\varepsilon) \, dy$ for all times $t$ (shifted left by $|L|\tilde{n}^{-1/(m+2)} +2\Pi^{1-\ve}(\ve)$ and up by $\eps$, or right by $R\tilde{n}^{-1/(m+2)} +\Pi^{1-\ve}(\ve)$, respectively). 

Next, because $v^{-,N}(x,0;\tilde{n}), v^{+,N}(x,0;\tilde{n})$ are discretizations of Lipschitz functions, there exists $\eta>0$ such that for if $\xmesh \leq \eta$, then for all $x$,  
\begin{align*}
&|v^{-,N}(x+\xmesh,0) - v^{-,N}(x,0)| < p^*,\\
&|v^{+,N}(x+\xmesh,0) - v^{+,N}(x,0)| < p^*.
\end{align*}

We claim that if $\xmesh \leq \min(\Pi^{1-\ve}(\ve),\tilde{n}^{-1/(m+2)})$, then $v^{+,N}(x,0) \geq v^N(x,0)$ for all $x$. To see this, first note that by choice of $L$, when $x < L\xmesh$, we have
$$v^N(x,0) \leq \varepsilon.$$

Because $v^{+,N}(x,0;\tilde{n}) \geq \varepsilon$ by definition, it follows that for $x < L\xmesh$, 
\begin{equation}\label{eq:upperleft}
v^N(x,0) \leq \varepsilon \leq v^{+,N}(x,0;\tilde{n}).
\end{equation}

For $x \geq L\xmesh$, notice that since $\xmesh \leq \tilde{n}^{-1/(m+2)}$ and $L \le 0$, we have $L(\xmesh-\tilde{n}^{-1/(m+2)})\ge 0$. Since $v_\varepsilon$ is nondecreasing, it follows that 
\begin{align*}
v^{+,N}(x,0;\tilde{n}) &= v_\varepsilon(\lfloor (x- L\tilde{n}^{1/(m+2)} + 2\Pi^{1-\ve}(\ve))/\xmesh \rfloor \xmesh,0) + \varepsilon\\
& \ge 
v_\varepsilon( \lfloor (L(\xmesh-\tilde{n}^{-1/(m+2)})+2\Pi^{1-\ve}(\ve))/\xmesh \rfloor \xmesh,0)\\
&\geq v_\varepsilon(\lfloor 2\Pi^{1-\ve}(\ve)/\xmesh\rfloor \xmesh,0)+\varepsilon\\
&\geq v_\varepsilon(\Pi^{1-\ve}(\ve),0) + \varepsilon = 1.
\end{align*}
The last inequality holds since $\lfloor 2\Pi^{1-\ve}(\ve)/\xmesh\rfloor \xmesh \ge \Pi^{1-\ve}(\ve)$ when $\xmesh \le \Pi^{1-\ve}(\ve)$. 

Therefore, we have that for $x \geq L\xmesh$ and $\xmesh \leq \min(\Pi^{1-\ve}(\ve),\tilde{n}^{-1/(m+2)})$, 
\begin{equation}\label{eq:upperright}
v^N(x,0) \leq 1 \leq v^{+,N}(x,0;\tilde{n}).
\end{equation}

Combining \eqref{eq:upperleft} and \eqref{eq:upperright}, we see that when $\xmesh \leq \min(\Pi^{1-\ve}(\ve),\tilde{n}^{-1/(m+2)})$, 
$$v^N(x,0) \leq v^{+,N}(x,0;\tilde{n})$$
for all $x$. 

Similarly, if $\xmesh \leq \tilde{n}^{-1/(m+2)}$, then for all $x < R\xmesh$, we have 
\begin{align*}
v^{-,N}(x,0;\tilde{n}) &= v_\varepsilon(\lfloor (x- R\tilde{n}^{-1/(m+2)} - \Pi^{1-\ve}(\ve))/\xmesh \rfloor \xmesh,0)\\
&<v_\varepsilon(R \xmesh - R\tilde{n}^{-1/(m+2)} -\Pi^{1-\ve}(\ve),0)\\
&< v_\varepsilon(-\Pi^{1-\ve}(\ve),0)\\
&= 0 \\
&\leq v^N(x,0).
\end{align*}

Moreover, if $x \geq R\xmesh$, then by the choice of $R$, 
$$v^{-,N}(x,0;\tilde{n}) \leq 1-\varepsilon \leq v^N(x,0),$$
where the first inequality comes from the fact that $v^{-,N}(x,0;\tilde{n}) \leq 1-\varepsilon$ for all $x$. Therefore, if $\xmesh \leq \tilde{n}^{-1/(m+2)}$, then for all $x$, 
$$v^{-,N}(x,0;\tilde{n}) \leq v^N(x,0).$$

Together, the above bounds give us an ordering of the initial conditions of the three schemes whenever $\xmesh \leq \min(\Pi^{1-\ve}(\ve),\tilde{n}^{-1/(m+2)})$:
$$v^{-,N}(x,0;\tilde{n}) \leq v^N(x,0) \leq v^{+,N}(x,0;\tilde{n}).$$
If $N$ is large enough that $\xmesh \leq \min(\Pi^{1-\ve}(\ve),\tilde{n}^{-1/(m+2)}, \eta)$, then we also have that $v^{-,N}(x,0;\tilde{n}),$ $v^N(x,0),$ $v^{+,N}(x,0;\tilde{n})$ are each bounded above by $p^*$ and ordered. 
For such $N$, it then follows by Lemma~\ref{lem:p*_monotonicity} that for all $x \in \R$, $t > 0$, 
\[
v^{-,N}(x,t;\tilde{n}) \leq v^N(x,t) \leq v^{+,N}(x,t;\tilde{n}).
\]

Now fix $T > 1$, $K\subseteq \R$ compact,  and $\delta > 0$. We can then apply Proposition \ref{p.intlip} at time $t = 1 < T$ to see that there exists a constant $N_0=N_0(m,\delta,M, K, T)$ such that for all $N \geq N_0$, for all $x\in\R$,
\begin{align*}
v^{+,N}(x,1;\tilde{n}) &\leq v^+(x,1;\tilde{n}) + \delta \hspace{.2in} \text{and}\\
v^{-,N}(x,1;\tilde{n}) & \geq v^-(x,1;\tilde{n}) - \delta.
\end{align*}
Combining the above inequalities, and recalling that $\xmesh=\xmesh^N=N^{-1/(m+2)}$, we see that if $N \geq N_0$ is large enough that $\xmesh \leq \min(\Pi^{1-\ve}(\ve),\tilde{n}^{-1/(m+2)}, \eta)$, then
$$\P(X_N \leq xN^{1/(m+2)}) = v^N(x,1) \in [v^-(x,1;\tilde{n}) - \delta, v^+(x,1;\tilde{n}) + \delta]\, .$$
For $\tilde{n} \geq \max(1/\Pi^{1-\ve}(\ve),N_0,\eta^{-(m+2)})$, requiring that $N \geq \tilde{n}$ will automatically satisfy the other constraints on $N$. 

Furthermore, by Theorem \ref{t.id}, we know that $v_{\ve}$ is the unique viscosity solution of 
\begin{equation*}
\begin{cases}
(v_\ve)_{t}-\frac{m+1}{2}\left|(v_{\ve})_{x}\right|^{m}(v_{\ve})_{xx}=0&\text{in $\RR\times (0,T]$,}\\
v_{\ve}(x,0)=v_{\varepsilon}(x,0)  &\text{in $\RR$,}
\end{cases}
\end{equation*}
and since the PDE is invariant to shifts in space and addition by constants, this implies that 
\begin{equation*}
v^+(x,t;\tilde{n})=\eps+ \int_{-\infty}^{x-L\tilde{n}^{-1/(m+2)}+2\Pi^{1-\ve}(\ve)} \hat{U}(y,t+\varepsilon;1-\varepsilon) \, dy,
\end{equation*}
and 
\begin{equation*}
 v^-(x,t;\tilde{n})=\int_{-\infty}^{x-R\tilde{n}^{-1/(m+2)}-\Pi^{1-\ve}(\ve)} \hat{U}(y,t+\varepsilon;1-\varepsilon) \, dy.
 \end{equation*}

The above bounds then imply that 
\begin{align*}
&\limsup_{N\to\infty}\P(X_N \leq xN^{1/(m+2)}) \\
&\leq  \inf_{\tilde{n} \geq \max(1/\Pi^{1-\ve}(\ve),N_0,\eta^{-(m+2)})} \left(\varepsilon + \int_{-\infty}^{x-L\tilde{n}^{-1/(m+2)}+2\Pi^{1-\ve}(\ve)} \hat{U}(y,1+\varepsilon;1-\varepsilon) \, dy \right) + \delta\\
&= \varepsilon + v_\varepsilon(x+2\Pi^{1-\ve}(\ve),1)+ \delta
\end{align*}
and 
\begin{align*}
&\liminf_{N\to\infty} \P(X_N \leq xN^{1/(m+2)}) \\
&\geq \sup_{\tilde{n} \geq \max(1/\Pi^{1-\ve}(\ve),N_0,\eta^{-(m+2)})}  \left(\int_{-\infty}^{x-R\tilde{n}^{-1/(m+2)}-\Pi^{1-\ve}(\ve)} \hat{U}(y,1+\varepsilon;1-\varepsilon) \, dy\right) - \delta \\
&=  v_\varepsilon(x-\Pi^{1-\ve}(\ve),1) - \delta.
\end{align*}

Now recalling by the explicit formula for $\Pi^{1-\ve}(\ve)$ in \eqref{eq:Pi} and the formula for $v_\varepsilon$ in \eqref{eq:vvar}, we have
$$\lim_{\varepsilon \to 0} \Pi^{1-\ve}(\ve) = 0$$
and
$$\lim_{\varepsilon \to 0} v_\varepsilon(x - \Pi^{1-\ve}(\ve),1) = \int_{-\infty}^x \hat{U}(y,1;1) \, dy
=\lim_{\varepsilon \to 0} v_\varepsilon(x + 2\Pi^{1-\ve}(\ve),1)\, .
$$

Therefore, taking the limit as $\delta, \varepsilon \to 0$, we see that 
\begin{align*}
\int_{-\infty}^x \hat{U}(y,1;1) \, dy &\leq \liminf_{N\to\infty} \P(X_N \leq xN^{1/(m+2)}) \\
&\leq \limsup_{N\to\infty} \P(X_N \leq xN^{1/(m+2)})\\
&\leq \int_{-\infty}^x \hat{U}(y,1;1) \, dy,
\end{align*}
and so 
$$\lim_{N\to\infty} \P\left(\frac{X_N}{N^{1/(m+2)}} \leq x\right) = \int_{-\infty}^x \hat{U}(y,1;1) \, dy.$$
The fact that $\hat{U}(y,1;1)$ is the density of the scaled and recentered Beta distribution \eqref{e.genbeta} was explained in our discussion in Section \ref{ss.zkb}; this completes the proof. 
\end{proof}

\section{Relaxation of $p^*$-bounded and the Proof of Theorem \ref{t.main}}\label{s.removep*}
\subsection{Eventually $p^{*}$-bounded distributions}
In this section, we show that every symmetric cooperative motion eventually has a $p^*$-bounded distribution.
\begin{prop}\label{p.removep*}
Let $(X_n, n \geq 0)$ be \scmmu-distributed. There exists a constant $N:=N(m)$ such that for all $n \geq N$, 
$\max_{k\in \Z} \P(X_n = k) \leq p^*$.
\end{prop}

\begin{proof}
By Lemma~\ref{lem:p*_monotonicity} we know that  if $\max_{k \in \Z} p^n_k \le p^*$ then also $\max_{k \in \Z} p^{n+1}_k \le p^*$.
It therefore suffices to show that there exists $C > 0$ such that for all $n \ge 0$ and $k \in \Z$, if $\max_{k\in \Z} p^{n}_{k} > p^*$ then
\begin{equation}\label{eq:onestep_dec}
p^{n+1}_k \le \max(p^n_{k-1},p^n_k,p^{n}_{k+1})  - C. 
\end{equation}
Indeed, if this holds, then for all $n \ge 0$ and all $k \in \Z$,
\begin{align*}
\P(X_{n+1} = k) 
& \le 
\max\big(p^*,\max(\P(X_n = k-1),\P(X_{n} = k),\P(X_{n} = k+1))-C\big) \\
& \le \max(p^*,\max_{k \in \Z} \P(X_n=k)-C)\, ,
\end{align*}
which implies that $\max_{k\in \Z} \P(X_n = k) \le p^*$ for all $n \ge (1-p^*)/C$.

We first treat the case that $\max(p^n_{k-1},p^n_k,p^{n}_{k+1})> p^* \ge 1/2$. This implies that $m \geq 1$, and we will use this fact in the argument. By definition we have 
\[
p^{n+1}_k - p^n_k
=
\frac{1}{2}\pran{
(p^n_{k+1})^{m+1}-2(p^n_{k})^{m+1}+(p^n_{k-1})^{m+1}}. 
\]
We can write the right-hand side as $-(z^{m+1}-\tfrac{1}{2}x^{m+1}-\tfrac{1}{2}y^{m+1})$ with $z=p^n_k,x=p^n_{k+1},y=p^n_{k-1}$. Since $z^{m+1}$ is increasing in $z$ and also $a^{m+1}+b^{m+1} \le (a+b)^{m+1}$ for $a,b \ge 0$ by convexity, 
if $x+y \le \tfrac{1}{2}\le z$, then $z^{m+1} \ge 2^{-(m+1)} \ge x^{m+1}+y^{m+1}$. This implies that 
\[
z^{m+1}-\frac{1}{2}x^{m+1}-\frac{1}{2}y^{m+1} \ge z^{m+1}-\frac{1}{2}z^{m+1}=\frac{1}{2}z^{m+1}\geq \frac{1}{2^{m+2}}\, , 
\]
which shows that if $p^n_k \ge 1/2$, then $p^{n+1}_k \le p^n_k - 1/2^{m+2}$.

Next, by definition we have 
\begin{align*}
p^{n+1}_{k}-p^n_{k+1} 
& = p^n_{k}-p^n_{k+1} + \frac{1}{2} 
\Big((p^n_{k+1})^{m+1} - 2(p^n_{k})^{m+1} + (p^n_{k-1})^{m+1}
\Big)\\
& = - \Big(z-\frac{1}{2} z^{m+1} - y + y^{m+1} - \frac{1}{2} x^{m+1}
\Big), 
\end{align*}
with $z=p^n_{k+1}, y=p^n_{k}, x=p^n_{k-1}$.
Moreover, since $x + y \leq 1/2$, $\min(x,y) < 1/3$. (Indeed, $\min(x,y) \le 1/4$, but there are enough factors of $2$ floating around that we instead use the bound $1/3$ to reduce potential confusion over what factors come from where.) Combined with the fact that when $m \geq 1$, $x - x^{m+1}$ and $x^{m+1}$ are increasing on $[0,1/2]$, we thus have  
\[
-y + y^{m+1} - \frac{1}{2}x^{m+1} \ge -\frac{1}{2}+\frac{1}{2^{m+1}} - \frac{1}{2^{m+2}} + C_1\, ,
\]
where $C_1 := \min\left(\frac{1}{2}-\left(\frac{1}{2}\right)^{m+1} - \frac{1}{3}+\left(\frac{1}{3}\right)^{m+1},\frac{1}{2}\left[\left(\frac{1}{2}\right)^{m+1}-\left(\frac{1}{3}\right)^{m+1}\right]\right)$. Since $z \ge \tfrac{1}{2}$ and  $z-\frac{z^{m+1}}{2}$ is concave down  on $[1/2,1]$, we also have $z-\frac{z^{m+1}}{2} \ge \tfrac{1}{2}-\tfrac{1}{2^{m+2}}$, so 
\[
z-\frac{1}{2} z^{m+1} - y + y^{m+1} - \frac{1}{2} x^{m+1}
\ge C_1 > 0,
\]
and therefore if $p^n_{k+1} \ge 1/2$ then $p^{n+1}_{k} \le p^n_{k+1}-C_1$. 

By a symmetric argument,  if $p^n_{k-1} \ge 1/2$ then $p^{n+1}_k \le p^n_{k-1} - C_1$. Combining the three cases, it follows that if $\max(p^n_{k-1},p^n_k,p^{n+1}_k) \ge p^* \ge 1/2$, then $p^{n+1}_k \le \max(p^n_{k-1},p^n_k,p^{n+1}_k)-\min(C_1,1/2^{m+2})$.

We now turn to the case that $p^* < 1/2$ and $\max(p^n_{k-1},p^n_k,p^{n}_{k+1}) > p^*$, and again show that in this case 
\[
p^{n+1}_k \le \max(p^n_{k-1},p^n_k,p^{n}_{k+1}) - C
\]
for some absolute constant $C > 0$. Whenever $p^{*}<\frac{1}{2}$, we automatically have $m<1$, and we will use this fact throughout the rest of the proof. The arguments are similar to those above, with the addition of new cases. 

Before beginning our analysis, we remark that 
\begin{equation}\label{e.inc}
\text{$x\mapsto x-x^{m+1}$ is increasing on $[0, p^{*})$ and decreasing on $(p^{*},1)$,}
\end{equation}
and for $m<1$, 
\begin{equation}\label{e.1/2inc}
\text{$x\mapsto x-\frac{1}{2}x^{m+1}$ is increasing on $[0, 1)$,}
\end{equation}
We will use this fact several times in the argument.

Suppose that $p^n_k > p^*$. First, note that since $p^* > 1/3$, at least one of $p^n_{k-1}$ and $p^n_{k+1}$ is less than $1/3$. By the symmetry of the recurrence, we may assume that $p^n_{k-1} < 1/3$. If also $p^n_k \ge p^n_{k+1}$, it follows that 
\begin{align*}
\max(p^n_{k-1},p^n_k,p^n_{k+1})-p^{n+1}_{k}
& = p^n_k - p^{n+1}_{k} \\
& = (p^n_k)^{m+1} 
- \frac12 (p^n_{k-1})^{m+1}
- \frac12 (p^n_{k+1})^{m+1}\\
& \ge (p^n_k)^{m+1} 
- \frac12 (p^n_{k-1})^{m+1}
- \frac12 (p^n_{k})^{m+1}\\
& \ge \frac12(p^*)^{m+1} - \frac12 \left(\frac{1}{3}\right)^{m+1} 
\end{align*}
so $p^{n+1}_k < \max(p^n_{k-1},p^n_k,p^n_{k+1}) - \tfrac{1}{2}((p^*)^{m+1}-(1/3)^{m+1})$. On the other hand, if 
$p^n_{k+1} > p^n_{k}$ then it follows that 
\begin{align*}
\max(p^n_{k-1},p^n_k,p^n_{k+1})-p^{n+1}_{k}
& = p^n_{k+1} - p^{n+1}_{k} \\
& = 
p^n_{k+1} - \frac12 (p^n_{k+1})^{m+1} - p^n_k + (p^n_k)^{m+1} - \frac12 (p^n_{k-1})^{m+1}\, .
\end{align*}
Since $p^n_{k+1} \ge p^*$, by \eqref{e.1/2inc} we have $p^n_{k+1} - \tfrac12 (p^n_{k+1})^{m+1} \ge p^* - \tfrac12 (p^*)^{m+1}$. Using \eqref{e.inc}, we have $-p^n_k + (p^n_k)^{m+1} \ge -p^*+(p^*)^{m+1}$. Finally, since $p^n_{k+1} > p^n_k > 1/3$, we have $\tfrac12(p^n_{k-1})^{m+1} < \tfrac12(1/3)^{m+1}$. Using all these bounds together yields that 
\[
\max(p^n_{k-1},p^n_k,p^n_{k+1})-p^{n+1}_{k}
\ge \frac{1}{2}\left((p^{*})^{m+1}-\left(\frac{1}{3}\right)^{m+1}\right),
\]
so again $p^{n+1}_k \le \max(p^n_{k-1},p^n_k,p^n_{k+1}) - \tfrac{1}{2}((p^*)^{m+1}-(1/3)^{m+1})$. 

Finally, suppose that $p^n_k < p^* < \max(p^n_{k-1},p^n_k,p^{n}_{k+1})$. 
As noted above, this implies that $m < 1$. We suppose without loss of generality that $\max(p^n_{k-1}, p^n_k, p^n_{k+1}) = p^n_{k+1}$. We argue in two cases, depending on whether or not $p^n_{k-1}> p^*$. In both cases, we have 
\begin{align}
\max(p^n_{k-1},p^n_k,p^{n}_{k+1}) - p^{n+1}_k &= p^n_{k+1} - p^{n+1}_k \notag\\
& = 
p^n_{k+1} - \frac12 (p^n_{k+1})^{m+1} - p^n_k + (p^n_k)^{m+1} - \frac12 (p^n_{k-1})^{m+1}\, ,\label{eq:pnk_diff}
\end{align} 
and we aim to bound the right-hand side away from zero from below.

If $p^n_{k-1} \leq p^*$, then since both $x-x^{m+1}$ and $x^{m+1}/2$ are increasing on $[0,p^*)$ and $\min(p^n_{k-1},p^n_k) < 1/3<p^*$, it follows that 
\begin{align*}
\max(p^n_{k-1},p^n_k,p^{n}_{k+1}) - p^{n+1}_k &\geq p^n_{k+1} - \frac12 (p^n_{k+1})^{m+1} - p^* + (p^*)^{m+1} - \frac12 (p^*)^{m+1} + C_2\\
& = p^n_{k+1} - \frac12 (p^n_{k+1})^{m+1} - p^* + \frac12(p^*)^{m+1} + C_2\, ,
\end{align*}
where we have taken
\[
C_2 := \min\left(p^* - (p^*)^{m+1} - \frac{1}{3} + \left(\frac{1}{3}\right)^{m+1}, \frac12(p^*)^{m+1} - \frac12\left(\frac{1}{3}\right)^{m+1}\right)> 0\,. 
\]
By \eqref{e.1/2inc}, since $p^{n}_{k+1}=\max(p^{n}_{k-1}, p^{n}_{k}, p^{n}_{k+1})$, it follows that
\begin{align*}
\max(p^n_{k-1},p^n_k,p^{n}_{k+1}) - p^{n+1}_k 
> C_2 > 0\, .
\end{align*} 

If $p^n_{k-1} > p^*$, then we bound the right-hand side of \eqref{eq:pnk_diff} away from zero by showing that in this case, $\max(p^n_{k-1},p^n_k,p^{n}_{k+1}) - p^{n+1}_k$ is minimized when $p^n_{k-1} = p^n_{k+1} = p^*$ and $p^n_k = 1-2p^*$. To accomplish this, we reparametrize the equation, letting $p^n_{k-1} + p^n_{k+1} = 2p^* + T$ and $p^n_{k+1} = p^* + \frac{T}{2} + \La$, so that $p^n_{k-1} = p^* + \frac{T}{2} - \La$. Notice that $T \in [0,1-2p^*]$ since $p^{n}_{k+1}=\max(p^{n}_{k-1}, p^{n}_{k}, p^{n}_{k+1})$, and $\La \in \left[0,T/2\right]$, since $p^{n}_{k+1}\geq p^{n}_{k-1}$. With these new variables, we have 
\begin{align*}
g(T,\La,p^n_k) &:= \max(p^n_{k-1},p^n_k,p^{n}_{k+1}) - p^{n+1}_k \\
&= p^* + \frac{T}{2} + \La - \frac12\left(p^* + \frac{T}{2} + \La\right)^{m+1} - \frac12\left(p^* + \frac{T}{2} - \La\right)^{m+1} - p^n_k + \left(p^n_k\right)^{m+1} \, . 
\end{align*}
It follows that 
\[
\frac{\partial g}{\partial \La}
=1-\frac{m+1}2(p^*+\tfrac{T}2+\La)^m+\frac{m+1}2(p^*+\tfrac{T}2-\La)^m > 0
\]
the last inequality holding for $\La \in [0,T/2]$ since $m < 1$. This implies that 
\[
g(T,\La,p^n_k) \ge g(T,0,p^n_k). 
\]
By \eqref{e.inc} and since $p^n_k \in [0,1-2p^*-T]\subset [0,p^*)$, it follows that 
\begin{align*}
g(T,\La,p^n_k) &\geq g(T,0,1-2p^*-T)\\
&= p^* + \frac{T}{2} - \left(p^* + \frac{T}{2} \right)^{m+1} - (1-2p^*-T) + \left(1-2p^*-T\right)^{m+1} \, .
\end{align*}  
Finally, taking a derivative of $g(T,0,1-2p^*-T)$ in $T$, we get
\begin{equation*}
\frac{dg(T,0,1-2p^*-T)}{dT} 
= \frac{3}{2}-\frac{m+1}{2}\left(p^*+\frac{T}{2}\right)^{m}- (m+1)\left(1-2p^*-T\right)^m. 
\end{equation*}
Since $p^{*}=\frac{1}{(m+1)^{1/m}}$, we have
\begin{align*}
\left.\frac{dg(T,0,1-2p^*-T)}{dT}\right|_{T = 0} &= \frac{3}{2}-\frac{m+1}{2}\left(p^*\right)^{m}- (m+1)\left(1-2p^*\right)^m \\
&= \frac{3}{2} - \frac{1}{2} - (m+1)\left(1-2p^*\right)^m \\
&= 1 - (m+1)\left(1-2p^*\right)^m > 0
\end{align*}
the inequality holding since $1-2p^*<p^*$  and $m<1$. To see that $\frac{dg(T,0,1-2p^*-T)}{dT} > 0$ for all $T \in [0,1-2p^*]$, we show that $\frac{d^2g(T,0,1-2p^*-T)}{dT^2} > 0$. We compute
\begin{align*}
\frac{d^2g(T,0,1-2p^*-T)}{dT^2} &= -\frac{m(m+1)}{4}\left(p^*+\frac{T}{2}\right)^{m-1} + m(m+1)(1-2p^*-T)^{m-1}\\
&= m(m+1)\left[(1-2p^*-T)^{m-1}-\frac{1}{4}\left(p^*+\frac{T}{2}\right)^{m-1}\right] > 0,
\end{align*}
where the last inequality is due to the fact that $1-2p^*-T<p^*+T/2$ and $x^{m-1}$ is decreasing on $[0,1]$. 

Because the second derivative is positive and $\tfrac{dg}{dT}|_{T=0} > 0$, we know that $\tfrac{dg}{dT} > 0$ for all $T \in [0,1-2p^*]$ and thus that $g(T,0,1-2p^*-T) \geq g(0,0,1-2p^*)$. Putting this all together, we get 
\begin{align*}
g(T,D,p^n_k) &\geq g(0,0,1-2p^*)\\
&= p^* - \left(p^* \right)^{m+1} - (1-2p^*) + \left(1-2p^*\right)^{m+1}  =:C_3 > 0\, ,
\end{align*}
the inequality holding by \eqref{e.inc}, since $p^{*}> 1-2p^{*}$. 
Combining the above cases, it follows that when $p^* < 1/2$ and $\max(p^n_{k-1},p^n_k,p^{n}_{k+1})>p^*$, we have 
$$p^{n+1}_k \le \max(p^n_{k-1},p^n_k,p^{n}_{k+1}) - C$$
for $C = \min(C_2, C_3)>0$, where $C_2,C_3$ are as defined above.
\end{proof}

\subsection{The proof of Theorem \ref{t.main}}\label{ss.proofofthm}
We conclude the section with the proof of Theorem \ref{t.main}.
\begin{proof}[Proof of Theorem \ref{t.main}]
Fix $m>0$ and $\mu$ an initial probability distribution. As $(X_{n}, n\geq 0)$ is $\scmmu$-distributed, by Proposition \ref{p.removep*}, there exists $N=N(m)$ so that for all $n\geq N$, $\max_{k\in \Z} \P(X_n = k) \leq p^*$. Define $\tilde{\mu}$ to be the distribution of $X_{N}$, so that $\tilde{\mu}$ is $p^*$-bounded. If $\tilde{X}_{n}:=X_{N+n}$, then $(\tilde{X}_{n}, n \geq0)$ is SCM$(m, \tilde{\mu})$-distributed. By Proposition \ref{p.goodic}, this implies that 
\begin{equation*}
\lim_{n\to\infty}\P\left(\frac{\tilde{X}_{n}}{n^{1/(m+2)}}\leq x\right)=v(x,1)
\end{equation*}
locally uniformly in $x$, where $v(x,1)$ is the CDF of 
\begin{equation*}
\frac{2^{(m+1)/(m+2)}(m+1)^{1/(m+2)}}{D^{m/(m+2)}\rho^{1/2}}\left[\Bet\left(\frac{m+1}{m}, \frac{m+1}{m}\right)-\frac{1}{2}\right].
\end{equation*}
Since $N$ is fixed and $v(x,1)$ is continuous, this implies that for all $x\in \R$, 
\begin{equation*}
\lim_{n\to\infty}\P\left(\frac{X_{n}}{n^{1/(m+2)}}\leq x\right)=\lim_{n\to\infty}\P\left(\frac{\tilde{X}_{n}}{n^{1/(m+2)}}\leq x\left(\frac{n-N}{n}\right)^{1/(m+2)}\right)=v(x,1),
\end{equation*}
so
\begin{equation*}
\frac{X_n}{n^{1/(m+2)}}\xrightarrow{d} \frac{2^{\frac{m+1}{m+2}}(m+1)^{1/(m+2)}}{D^{\frac{m}{m+2}}\rho^{1/2}}\left(B-\frac{1}{2}\right),
\end{equation*}
where $B$ is $\Bet\left(\frac{m+1}{m}, \frac{m+1}{m}\right)$-distributed, as required. 
\end{proof}

\section{Generalizations and Extensions}\label{s.conclusion}
\subsection{Robustness of the Method}\label{s.robust}

While we have focused on proving Theorem \ref{t.main}, we believe the approach and method of proof are rather general, and can potentially be applied to other interesting probabilistic models. We next present a summary of the general ingredients needed, by our method, to prove a distributional convergence result for a generic sequence of random variables $(X_{n}, n\geq 0)$ on $\Z$, whose laws are specified by a recursive distributional equation (RDE). Throughout the discussion, we highlight where each of these ingredients appears in our proof of the Bernoulli central limit theorem in Section \ref{sec:casestudy} and the proof of Theorem \ref{t.main}. 

Given an (RDE) for the CDF $F^{n}_{k}=\P[X_{n}\leq k]$, we seek the following ingredients: 
\begin{enumerate}
\item \emph{The RDE is a discrete approximation of a well-posed PDE}. Presented with an RDE, one may attempt to ``guess'' a PDE which the RDE approximates. It is then crucial to identify whether this PDE is well-posed; in particular whether it is equipped with a solution theory which has existence and uniqueness for initial conditions which are (extended) CDFs. In the context of the Bernoulli central limit theorem, we relied on the theory of classical solutions; in the context of Theorem \ref{t.main}, we relied on viscosity solutions for PDEs with Lipschitz continuous, extended CDF initial conditions. 
\item \emph{A self-similar scaling}. Since the RDE describes the dynamics at integer times and integer locations, one further needs to identify a self-similar scaling of space and time which keeps the PDE limit of the CDF invariant. In the case of the Bernoulli central limit theorem, this was the classical Brownian/parabolic scaling, and in the case of Theorem \ref{t.main}, this was given by $t \mapsto ct$, $x \mapsto c^{1/(m+2)}x$. 
Upon identifying a self-similar scaling, one can hope to translate the RDE into a finite difference scheme, as we did in transitioning from \eqref{e.plap1} to \eqref{e.plap2}. 
\item \emph{An established theory of convergence results for finite difference schemes of the PDE.} In these works, we relied on the convergence results of Barles and Souganidis \cite{BS} for degenerate parabolic PDEs with continuous initial conditions; equipped with convergence results, we had to ensure that our finite difference scheme satisfied the appropriate hypotheses of the Barles-Souganidis theorem. For other models, different convergence results for finite difference schemes approximating solutions of PDEs may be used. (For example, \cite{MR4391736} uses results of Crandall and Lions \cite{Crandall1984}; \cite{Addario-Berry2019} uses results of Evje and Karlsen \cite{ek}.)
\item \emph{Coupling different initial conditions.} It is possible that either the solution theory of the well-posed PDE or the convergence results for the associated finite difference schemes are not robust enough to handle discretizations of a PDE with a Heaviside initial condition. This was the case in both the Bernoulli central limit theorem and Theorem \ref{t.main}. We resolved this by developing a ``sandwiching argument,'' which ultimately relied on coupling stochastic processes with different initial conditions. 
\item \emph{(Optional.) A way of identifying the explicit solution of the PDE with a Heaviside initial condition}. The prior ingredients yield that the CDF of a rescaled random variable converges to a function which solves a PDE with a Heaviside initial condition. This function (evaluated at time $1$) is precisely the CDF of the limiting random variable. 
To obtain a more explicit description of the limiting distribution, one may attempt to identify an explicit representation for the solution of this PDE with a Heaviside initial condition. For the Bernoulli central limit theorem, this was automatic due to the theory of fundamental solutions. In the setting of Theorem \ref{t.main}, we needed to develop the results of Section \ref{s:ZKB} in order to identify the symmetric Beta random variable in the limit. 
\end{enumerate}

While this discussion is not a precise general theorem, we hope that it clarifies the main ingredients needed in order to adapt our approach to different models. 

\smallskip
One such adaptation we mention here is cooperative motion with ``$q$-lazy dynamics'', where as before, 
\begin{equation*}
X_{n+1} = \begin{cases} 
			X_n + E_n	& \mbox{ if } X_n=X_{n,1}=\ldots=X_{n,m} \\
			X_n	&\mbox{ otherwise,} 
			\end{cases}
\end{equation*}
for $X_{n,1},\ldots,X_{n,m}$ IID copies of $X_n$, but now $(E_n,n \ge 0)$ are IID with $\p{E_0=1}=q/2=\p{E_0=-1}$, $\p{E_{0}=0}=1-q$. In this case, a straightforward extension of our method yields that the appropriate PDE governing the limiting probability distribution is given by 
\begin{equation*}
u_{t}-\frac{q}{2}(u^{m+1})_{xx}=0. 
\end{equation*}
An analysis of the corresponding ZKB solution of the above PDE yields that
\begin{equation*}
\frac{X_n}{n^{1/(m+2)}}\xrightarrow{d} \frac{2^{\frac{m+1}{m+2}}(m+1)^{1/(m+2)}q^{1/(m+2)}}{D^{\frac{m}{m+2}}\rho^{1/2}}\cdot  \left(B-\frac{1}{2}\right),
\end{equation*}
where $B$ is again $\Bet\left(\frac{m+1}{m},\frac{m+1}{m}\right)$-distributed. 
 
It would of course be natural to consider step sizes with support not contained in $\{-1,0,1\}$; however in the setting of cooperative motion, monotonicity becomes an issue for more general step sizes. (This is explained in more detail in \cite[Section 5.2]{MR4391736}.) There is only one other family of cases we can handle --- when the step sizes have support contained in $\{-R,0,R\}$ for some fixed $R \in \N$. We turn to this in the next subsection.

\subsection{Persistent Lattice Effects}
Fix $q \in (0,1]$ and $r \in [1/2,1]$, and 
consider a variant of the cooperative motion process defined by 
\begin{equation}\label{eq:main_recurrence_generalized}
X_{n+1} = \begin{cases} 
			X_n + E_n	& \mbox{ if } X_n=X_{n,1}=\ldots=X_{n,m} \\
			X_n	&\mbox{ otherwise} 
			\end{cases}
\end{equation}
where $(E_n,n \ge 0)$ are IID with 
$\p{E_0=R}=rq$, $\p{E_0=-R}=(1-r)q$ and $\p{E_0=0}=1-q$. 
(As before, $X_{n,1},\ldots,X_{n,m}$ are IID copies of $X_n$.) If the initial distribution $\mu$ is supported by $R\Z+k$ for some $1 \le k \le R$ then this process simply looks like a cooperative motion with $\{-1,0,1\}$ steps (which is symmetric if $r=1/2$ and asymmetric otherwise). However, for other initial distributions, alternative asymptotic behaviour can appear. The reason for this is explained in detail in the setting of {\em asymmetric} cooperative motion in our previous work \cite{MR4391736}, in which we prove the following result.
\begin{thm}[\cite{MR4391736}, Theorem 5.1]
\label{thm:lattice_old}
Consider the cooperative motion process defined in (\ref{eq:main_recurrence_generalized}), with $r > 1/2$ and with $r$ and $q$ not both $1$. Let $\pi_k = \p{X_0=k\mod R}$ for $k \in \{1,2,\ldots,r\}$. 
Then
\[
\frac{X_n}{n^{1/(m+1)}} 
\xrightarrow{d}
R(m+1) \left(\frac{(2r-1)q }{m^m}\right)^{\frac{1}{m+1}}  B \cdot \sum_{k=1}^R (\pi_k)^{\frac{m}{m+1}}\indc_{\left\{A=k\right\}},
\]
where $A$ is a random variable with $\P(A=k)=\pi_k$ for $k \in \{1,2,\ldots,R
\}$, and $B$ is $\Bet(\tfrac{m+1}{m},1)$-distributed and independent of $A$. 
\end{thm}
The analogous result for symmetric cooperative motion is as follows.
\begin{thm}\label{thm:lattice}
Consider the cooperative motion process defined in (\ref{eq:main_recurrence_generalized}), with $r = 1/2$. Let $\pi_k = \p{X_0=k\mod R}$ for $k \in \{1,2,\ldots,r\}$. Then 
\[
\frac{X_n}{n^{1/(m+2)}}\xrightarrow{d} \frac{2^{\frac{m+1}{m+2}}(m+1)^{\frac{1}{m+2}}q^{\frac{1}{m+2}}}{D^{\frac{m}{m+2}}\rho^{1/2}}\left(B-\frac{1}{2}\right)
\cdot \sum_{k=1}^R (\pi_k)^{\frac{m}{m+1}}\indc_{\left\{A=k\right\}},
\]
where $A$ is a random variable with $\P(A=k)=\pi_k$ for $k \in \{1,2,\ldots,R
\}$, and $B$ is $\Bet(\tfrac{m+1}{m},\tfrac{m+1}{m})$-distributed and independent of $A$. 
\end{thm}
We omit the proof of Theorem~\ref{thm:lattice} as it is both straightforward and precisely analogous to the (rather short) proof of Theorem~\ref{thm:lattice_old} from \cite{MR4391736}.

\appendix
\section{Facts about Relevant Solutions} \label{s.app}

We begin with a review of distributional solutions. 
\begin{define}\label{d.ent}
Let $\Psi: \R\rightarrow \R$ be continuous and nondecreasing and let $u_{0}\in L^{1}(\R)\cap L^{\infty}(\R)$. We say that $u\in C([0,T]; L^{1}(\R))\cap L^{\infty}(\RR\times [0,T])$ is a distributional solution of 
\begin{equation}\label{e.genent}
\begin{cases}
u_{t}-(\Psi(u))_{xx}=0&\text{in $\R\times (0,T]$}\, , \\
u(x,0)=u_{0}(x)&\text{in $\R$}\, , 
\end{cases}
\end{equation}
if and only if for all $\vp\in C^{\infty}_{c}(\R\times [0,T))$, 
\begin{equation}\label{e.distsol}
\int_{0}^{T}\int_{-\infty}^{\infty}u\vp_{t}+\Psi(u)\vp_{xx}\, dxdt+\int_{-\infty}^{\infty} u_{0}(x)\vp(x,0)\, dx=0. 
\end{equation}
\end{define}

We also define viscosity solutions of a certain class of PDEs.

\begin{define}\label{d.visc}
Consider the good initial value problem
\begin{equation}\label{e.genvisc}
\begin{cases}
v_{t}+G(v_{x}, v_{xx})=0&\text{in $\R\times (0, T]$,}\\
v(x,0)=v_{0}(x)&\text{in $\R$}. 
\end{cases}
\end{equation}
We say that a bounded, uniformly continuous function $v: \R\times [0, T]\rightarrow \R$ is a viscosity {\em subsolution} of \eqref{e.genvisc} if for any function $\vp\in C^{2,1}(\RR\times (0, T])$ such that $v-\vp$ has a local maximum at $(x_{0},t_{0})\in \R\times (0, T]$, we have 
\begin{equation*}
\vp_{t}(x_{0},t_{0})+G(\vp_{x}(x_{0}, t_{0}), \vp_{xx}(x_{0}, t_{0}))\leq 0.
\end{equation*}
We say that $v$ is a viscosity {\em supersolution} of \eqref{e.genvisc} if for any function 
$\vp\in C^{2,1}(\RR\times (0, T])$ such that $v-\vp$ has a local minimum at $(x_{0},t_{0})\in \R\times (0, T]$, we have
\begin{equation*}
\vp_{t}(x_{0},t_{0})+G(\vp_{x}(x_{0}, t_{0}), \vp_{xx}(x_{0}, t_{0}))\ge 0.
\end{equation*}
Finally, we say that $v$ is a viscosity solution of \eqref{e.genvisc} if and only if $v$ is both a viscosity subsolution and supersolution of \eqref{e.genvisc}, and additionally, for all $x \in \R$, $v(y,t) \to v_0(x)$ as $(y,t) \to (x,0)$.

\end{define}

\begin{remark}\label{r.smooth}
If $u$ and $v$ belong to $C^{2,1}(\R\times (0, T])\cap C(\R\times [0, T])$ and are classical solutions of \eqref{e.genent} and \eqref{e.genvisc} on $\R\times(0,T]$, respectively, then they are also distributional and viscosity solutions, respectively. To see this, first note that in the case of \eqref{e.genent}, if $u$ is sufficiently regular then \eqref{e.distsol} is automatically satisfied by integration by parts. 

In the case when $v\in C^{2,1}(\R\times (0, T])\cap C(\R\times [0, T])$, verifying that $v$ is a viscosity solution is also straightforward since $G$ is degenerate elliptic. 
Whenever $v-\vp$ has a local maximum at $(x_{0}, t_{0}) \in \R\times(0,T)$, we have 
\begin{equation*}
v_{t}(x_{0}, t_{0})=\vp_{t}(x_{0}, t_{0})\quad\text{and}\quad v_{x}(x_{0}, t_{0})=\vp_{x}(x_{0}, t_{0})\quad\text{and}\quad v_{xx}(x_{0}, t_{0})\leq \vp_{xx}(x_{0}, t_{0}).
\end{equation*}
If $(x_{0}, t_{0})\in \R\times \left\{t=T\right\}$, we have 
\begin{equation*}
v_{t}(x_{0}, t_{0})\geq \vp_{t}(x_{0}, t_{0})\quad\text{and}\quad v_{x}(x_{0}, t_{0})=\vp_{x}(x_{0}, t_{0})\quad\text{and}\quad v_{xx}(x_{0}, t_{0})\leq \vp_{xx}(x_{0}, t_{0}).
\end{equation*}
This implies that 
\begin{equation*}
\vp_{t}(x_{0}, t_{0})+G(\vp_{x}(x_{0}, t_{0}), \vp_{xx}(x_{0}, t_{0}))\leq v_{t}(x_{0}, t_{0})+G(v_{x}(x_{0}, t_{0}), v_{xx}(x_{0}, t_{0}))=0,
\end{equation*}
and this verifies the subsolution property. The supersolution property is analogous. 
\end{remark}

A similar computation to the one above can be used to show that the ZKB solution shifted by any $\ve>0$ amount of time (so given by $\hat{U}(x,t+\eps; 1)$, where $\hat{U}$ is as in \eqref{e.Vrdef})
is in fact the unique distributional solution of the porous medium equation with initial condition $\hat{U}(x,\eps;1)$. 
\begin{lemma}\label{l.barenblatt}
Let
\begin{equation*}
\hat{U}(x,t; \theta) = \frac{1}{t^{1/(m+2)}} \left[\left(\frac{\sqrt{2}\theta}{D\sqrt{m+1}}\right)^{\frac{2m}{m+2}}-\frac {2\rho|x|^2}{(m+1)t^{2/(m+2)}}\right]^{1/m}_+\, .
\end{equation*}
For every $\ve>0$, the function $(x,t) \mapsto\hat{U}(x,t+\ve; \theta)$ is the unique distributional solution of 
\begin{equation}\label{e.pmeapp}
\begin{cases}
u_{t}-\frac{1}{2}(u^{m+1})_{xx}=0&\text{in $\R\times (0, \infty)$,}\\
u(x,0)=\hat{U}(x, \ve; \theta)&\text{in $\R$.}
\end{cases}
\end{equation}

\end{lemma}

\begin{proof}
Throughout the proof, to be consistent with our prior notation, we let $\Psi(r)=\frac{1}{2}r^{m+1}$. Define
\begin{multline*}
[-\Pi^{\theta}(t), \Pi^{\theta}(t)]\\
:=\left[-\frac{\theta^{\frac{m}{m+2}}(m+1)^{1/(m+2)}}{2^{1/(m+2)}\rho^{1/2}D^{m/(m+2)}}t^{1/(m+2)}, \frac{\sqrt{\theta}(m+1)^{1/(m+2)}}{2^{1/(m+2)}\rho^{1/2}D^{m/(m+2)}}t^{1/(m+2)}\right],
\end{multline*}
which is the support of $\hat{U}(\cdot, t; \theta)$. A routine calculation verifies that $\hat{U}(\cdot \,, \cdot\,; \theta)$ solves \eqref{e.pmeapp} for all $t>0$ and $x\in (-\Pi^{\theta}(t),\Pi^{\theta}(t))$. For $x\notin (-\Pi^{\theta}(t),\Pi^{\theta}(t))$ we have $\hat{U}(x,t; \theta)= 0$, and hence $\hat{U}(\cdot \,, \cdot\,; \theta)$ is also a classical solution on $\{(x,t):x \not\in[-\Pi^{\theta}(t),\Pi^{\theta}(t)]\}$.

This implies that $(x,t)\mapsto\hat{U}(x, t+\ve; \theta)$ is a piecewise smooth function which satisfies the porous medium equation in each of its subdomains of smoothness. Moreover, letting $u(x,t):=\hat{U}(x, t+\ve; \theta)$, we have that $u(\cdot, t)$ and $\left(\frac{1}{2}u^{m+1}(\cdot, t)\right)_{x}$ are both 0 at the endpoints of the support, $\left\{-\Pi^{\theta}(t+\ve), \Pi^{\theta}(t+\ve)\right\}$. By a standard integration by parts calculation, this implies that $u$ is a distributional solution.
\end{proof}

We conclude the appendix by stating the stability property of viscosity solutions. 
\begin{prop}\label{p.stable}\cite[Theorem 8.3]{primer} Let $(v_{n}, 1\leq n<\infty)$ denote a collection of viscosity subsolutions of 
\begin{equation*}
(v_{n})_{t}+G_{n}\left((v_{n})_{x},(v_{n})_{xx} \right)\leq 0\quad\text{in $\R\times (0, T]$},
\end{equation*}
where $G_{n}$ is degenerate elliptic and continuous for each $n$. If $v_{n}\rightarrow \bar{v}$ and $G_{n}\rightarrow \bar{G}$ locally uniformly, then $\bar{v}$ is a viscosity subsolution of
\begin{equation*}
\bar{v}_{t}+\bar{G}\left(\bar{v}_{x},\bar{v}_{xx} \right)\leq0\quad\text{in $\R\times (0, T]$}.
\end{equation*}
The analogous statement holds for viscosity supersolutions.
\end{prop}

\bibliographystyle{acm}
\bibliography{general_HRW}

\end{document}